\documentclass[10pt]{amsart} 
\usepackage{amsmath,amssymb,amsthm,amsfonts,amsopn,cite,mathrsfs}

\usepackage{color,geometry}

\theoremstyle{plain}
\newtheorem{theorem}{Theorem}[section]
\newtheorem{proposition}[theorem]{Proposition}
\newtheorem{lemma}[theorem]{Lemma}
\newtheorem{corollary}[theorem]{Corollary}
\theoremstyle{definition}
\newtheorem{definition}{Definition}[section]
\theoremstyle{remark}
\newtheorem{remark}{Remark}[section]
\newtheorem{example}{Example}[section]

\newcommand{\bbC}{{\mathbb C}}   
\newcommand{\CC}{{\mathbb{C}}}

\def\Cessum{{\lower 1.5pt \hbox{\large C}}{-}\!\!\!\sum}
\def\CHI{\hbox{\raise .5ex \hbox{$\chi$}}}
\def\esssup{\operatornamewithlimits{ess\,sup}}
\newcommand{\GLsp}{{\boldsymbol G\kern-0.1em\boldsymbol L}}

\newcommand{\inner}[2]{\langle #1,#2\rangle}   
\def\LiR{{\Lisp(\Rst)}}   

\def\LtR{{\Ltsp(\Rst)}}   

\def\Msp{{\boldsymbol M}}   
\def\mv1{\Msp_v^1}   

\def\NN{\mathbb{N}}
\newcommand{\norm}[1]{\lVert#1\rVert}   
\newcommand{\RR}{\mathbb{R}}   

\def\sabs{{{\raise 0.5pt \hbox{\footnotesize $|$}}}}

\newcommand{\sgn}{\operatorname{sgn}}   
\def\shah{\sqcup \hspace{-0.15em}\sqcup}   
\def\shah1{{\makebox[2.3ex][s]{$\sqcup$\hspace{-0.15em}\hfill $\sqcup$}}}   

\def\slb{{{\raise 0.5pt \hbox{\footnotesize $[$}}}}
\def\slcb{{{\raise 0.5pt \hbox{\footnotesize $\{$}}}}
\def\slp{{{\raise 0.5pt \hbox{\footnotesize $($}}}}   

\newcommand{\snorm}{{{\raise 0.5pt \hbox{\footnotesize $\|$}}}}   

\def\srb{{{\raise 0.5pt \hbox{\footnotesize $]$}}}}
\def\srcb{{{\raise 0.5pt \hbox{\footnotesize $\}$}}}}
\def\srp{{{\raise 0.5pt \hbox{\footnotesize $)$}}}}

\def\supp{\operatorname{supp}}   

\def\UU{ {\bf U}}   
\newcommand{\ZZ}{\mathbb{Z}}

\newcommand{\bd}{\mathbf}   

\renewcommand{\LtR}{\mathbf L^2(\mathbb{R})}
\renewcommand{\LiR}{\mathbf L^1(\mathbb{R})}
\newcommand{\AAi}{\mathcal A_1}
\newcommand{\AAm}{\mathcal A_m}
\renewcommand{\UU}{\mathcal U}

\newcommand{\Eye}{\operatorname{E}_{y,\epsilon}}

\def\esssup{\mathop{\operatorname{ess~sup}}}

\def\supp{\mathop{\operatorname{supp}}}
\newcommand{\abs}[1]{|#1|}
\makeatletter 
\newcommand{\LeftEqNo}{\let\veqno\@@leqno}

\DeclareFontFamily{U}{mathx}{\hyphenchar\font45}
\DeclareFontShape{U}{mathx}{m}{n}{
      <5> <6> <7> <8> <9> <10>
      <10.95> <12> <14.4> <17.28> <20.74> <24.88>
      mathx10
      }{}
\DeclareSymbolFont{mathx}{U}{mathx}{m}{n}
\DeclareFontSubstitution{U}{mathx}{m}{n}
\DeclareMathAccent{\widecheck}{0}{mathx}{"71}
\DeclareMathAccent{\wideparen}{0}{mathx}{"75}
\newcommand{\LtD}{\mathbf L^2(D)}

\newcommand{\oscUGd}{\text{osc}_{\UU^\delta,\Gamma}}

\begin{document}

 \author{Nicki Holighaus$^\dag$}
 \address{$^\dag$ Acoustics Research Institute Austrian Academy of Sciences, Wohllebengasse 12--14, A-1040 Vienna, Austria}
 \email{nicki.holighaus@oeaw.ac.at, peter.balazs@oeaw.ac.at}
 \author{Christoph Wiesmeyr$^*$}
 \address{$^*$ NuHAG, Faculty of Mathematics, Oskar-Morgenstern-Platz 1, A-1090 Vienna, Austria and \newline AIT Austrian Institute of Technology GmbH, Donau-City-Straße 1, A-1220 Vienna, Austria}
 \email{christoph.wiesmeyr.fl@ait.ac.at}
 \author{Peter Balazs$^\dag$}
 
 \title[Warped time-frequency representations, Part II: Integral transforms]{Construction of warped time-frequency representations on nonuniform frequency scales, Part II: Integral transforms, function spaces, atomic decompositions and Banach frames}
 
 \begin{abstract}
   We present a novel family of continuous linear time-frequency transforms adapted to a multitude of (nonlinear) frequency scales. 
   Similar to classical time-frequency or time-scale representations, the representation coefficients are obtained as inner 
   products with the elements of a continuously indexed family of time-frequency atoms. 
   These atoms are obtained from a single prototype function, by means of modulation, translation and warping. 
   By warping we refer to the process of nonlinear evaluation according to a bijective, increasing function,
   the warping function. 
    
   Besides showing that the resulting integral transforms fulfill certain basic, but essential properties, 
   such as continuity and invertibility, we will show that a large subclass of warping functions gives 
   rise to families of generalized coorbit spaces, i.e. Banach spaces of functions whose representations possess a certain localization. 
   Furthermore, we obtain sufficient conditions for subsampled warped time-frequency systems
   to form atomic decompositions and Banach frames. 
   To this end, we extend results previously presented by Fornasier and Rauhut
   to a larger class of function systems via a simple, but crucial modification.
   
   The proposed method allows for great flexibility, but by choosing 
   particular warping functions we also recover classical time-frequency representations, 
   e.g. $F(t) = ct$ provides the short-time Fourier transform and $F(t)=\log_a(t)$ provides wavelet transforms.
   This is illustrated by a number of examples provided in the manuscript.
 \end{abstract}

 \maketitle
 
 

\section{Introduction}

  This is the second part of a $2$ manuscript series about the properties of a 
novel family of time-frequency representations, see~\cite{howi14}. In this paper, we introduce the 
notion of (continuous) warped time-frequency transforms, a class of integral transforms 
representing functions in phase space with respect to nonlinear frequency 
scales. The goals of this contribution are showing the following properties: 
  \begin{enumerate}
   \item[(a)] These transforms possess some basic, but central properties, 
namely continuity and invertibility in a Hilbert space setting. The latter is 
obtained through a variant of Moyal's formula~\cite{mo49,gr01}.
   \item[(b)] They give rise to classes of generalized coorbit spaces, i.e. 
nested Banach spaces of functions with a certain localization in the associated 
phase space.
   \item[(c)] They are stable under a sampling operation, yielding atomic 
decompositions and Banach frames of warped time-frequency systems.
  \end{enumerate}
  In order to prove item (c), we will introduce a a slight modification to the 
discretization theory for generalized coorbit spaces presented in~\cite{fora05}, 
cf. Section \ref{sec:genosc}, 
enabling 
discretization results for our own construction. The first part of this 
series~\cite{howi14} investigates the construction of (Hilbert space) frames by 
means of discrete warped time-frequency systems. 
  
  In the last decades, time-frequency representations, in particular short-time 
Fourier~\cite{ga46,gr01} and wavelet~\cite{da92,dagrme86} transforms, have become 
indispensable tools in many areas from theoretical and applied 
mathematics to physics and signal processing. 
These classical time-frequency schemes measure the time-frequency distribution 
of a function as the correlation of that function with a family of 
time-frequency atoms. These atoms originate from the application of a set of 
unitary operators, translations and modulations in the short-time Fourier 
transform (STFT) and translations and dilations in the wavelet transform (WT), 
to a prototype function or \emph{mother wavelet}. By the uncertainty 
principle~\cite{fosi97,havin1994uncertainty}, no mother wavelet can be arbitrarily concentrated in time 
and frequency simultaneously and thus the choice of the prototype function 
completely determines the time-frequency trade-off of the representation, i.e. 
constant resolution in the case of STFT 
and resolution strictly proportional to the center frequency for wavelets. 
  
  This rigidity of classical time-frequency systems, particularly the fixed 
resolution of the STFT, has lead to the development of more general schemes for 
extracting time-frequency information from a function. Some prominent examples 
include the \emph{$\alpha$-transform}~\cite{fefo06,cofe78,fo89,hona03}, 
sometimes also referred to as flexible Gabor-Wavelet transform, in the 
continuous world and \emph{generalized shift-invariant 
systems}~\cite{rosh04,helawe02}, also known as nonuniform (analysis) filter banks in 
the signal processing community. Also of note are the countless variations on 
and extensions of the wavelet scheme, including but not limited to wavelet 
packets~\cite{coifman1994signal}, shearlets~\cite{kula12-alt,dakustte09}, curvelets~\cite{candes2005continuous} 
and ridgelets~\cite{candes1999ridgelets}. The previously mentioned transforms rely on the 
variation of resolution along frequency. The equivalent concept for variation 
along time are nonstationary Gabor 
systems~\cite{badohojave11,doma12,doma13-1,ho14-1}, which consider semi-regular 
modulations of a family of prototype functions that vary over time.
  
  Many applications of time-frequency representations require the transform used 
to be invertible, or more specifically bounded and boundedly invertible. The 
appropriate tool for the analysis of invertibility properties of 
time-frequency systems on Hilbert spaces is frame theory~\cite{dusc52,ch03}, 
given a countable family of time-frequency atoms. For uncountable families the theory of continuous 
frames~\cite{alanga93,alanga00} is appropriate. Whenever a family of time-frequency 
atoms $\Phi$ forms a (discrete or continuous) frame for a Hilbert space 
$\mathcal H$, then the following automatically hold:
  \begin{itemize}
   \item Every function $f\in \mathcal H$ is uniquely determined by its inner 
product with the frame elements, and
   \item every function $f\in \mathcal H$ can be written as a superposition of 
the frame elements with norm-bounded coefficients.
  \end{itemize}
  If we desire to analyze or decompose functions contained not in a Hilbert, but 
in a Banach space $\mathcal B$, then the two properties above cease to be 
equivalent. In that case, we have to determine separately whether $\Phi$ forms a 
  Banach frame and/or atomic decomposition~\cite{gr91,tr92} for 
$\mathcal B$. Where applicable, coorbit theory and its various 
generalizations~\cite{fegr89,fegr89-1,gr91,fora05} yield the appropriate Banach spaces for such an 
analysis, see below.
  
  In this contribution, we introduce a novel family of time-frequency representations 
  adapted to nonlinear frequency scales. Uniquely determined by the choice of a single prototype
  atom and a warping function that determines the desired frequency scale, our construction 
  provides a family of time-frequency atoms with uniform frequency resolution \emph{on the chosen frequency scale}.
  For particular choices of the warping function, we recover the continuous short-time Fourier and wavelet transforms.
  Hence, the proposed \emph{warped time-frequency representations} can be considered a unifying framework for a
  larger class of time-frequency systems.
  
  We show that the transforms in this family provide continuous representations (considered as functions in phase space). 
  Furthermore, the proposed systems form continuous tight frames, even satisfying orthogonality relations similar to 
  Moyal's formula~\cite{mo49,gr01} for the STFT.
  
  We obtain coorbit spaces associated to each frequency scale, i.e. classes of Banach spaces that classify the time-frequency 
  behavior of a function in terms of the corresponding warped time-frequency representation. Through a minor extension of the generalized 
  coorbit theory by Fornasier and Rauhut~\cite{fora05}, we can also prove sufficient conditions for countable subfamilies 
  of warped time-frequency atoms to form Banach frames and atomic decompositions for these coorbit spaces.
  
  \subsection{Related work}
  
   The idea of a logarithmic warping of the frequency axis to obtain wavelet systems from a system of translates is not entirely new and was first used in the proof of the so called painless conditions for wavelets systems~\cite{dagrme86}. However, the idea has never been relaxed to other frequency scales so far.
 While the parallel work by Christensen and Goh \cite{chgo13} focuses on exposing the duality between Gabor and wavelet systems via the mentioned logarithmic warping, we allow for more general warping functions to generate time-frequency transformations beyond wavelet and Gabor systems.
 The warping procedure we propose has already proven useful in the area of graph signal processing~\cite{hoshvawi13}.
 
 A number of methods for obtaining \emph{warped time frequency representations} have been proposed, e.g. by applying a unitary basis transformation to Gabor or wavelet atoms \cite{bajo93-1,ba94-2,caev98,doevma13}. Although unitary transformations bequeath basis (or frame) properties to the warped atoms, the warped system provides an undesirable, irregular time-frequency tiling, see \cite{caev98}.
 
 Closer to our own approach, Braccini and Oppenheim~\cite{brop74}, as well as Twaroch and Hlawatsch~\cite{hltw98}, propose a warping of filter bank transfer functions only, by defining a unitary warping operator. However, in ensuring unitarity, the authors give up the property that warping is shape preserving when observed on the warped frequency scale. In this contribution, we trade the unitary operator for a shape preserving warping.
 
 A more traditional approach trying to bridge the gap between the linear frequency scale of the short-time Fourier transform and the logarithmic scale associated to the wavelet transform is the \emph{$\alpha$-transform}~\cite{fefo06,cofe78,fo89,hona03} that employs translation, modulation and dilation operators with a fixed relation between modulation and dilation, determined by the parameter $\alpha\in [0,1]$. For $\alpha = 0$, the short-time Fourier transform is obtained, while the limiting case $\alpha = 1$ provides a system with logarithmic frequency scale similar, but not equivalent, to a wavelet system. Both our construction and the $\alpha$-transform can be considered special cases inside the framework of continuous nonstationary Gabor transforms, see \cite{spexxl14}, or the equivalent generalized translation-invariant systems~\cite{jale14}.
 
 Coorbit theory and discretization results for time-frequency systems on Banach spaces date back to the seminal work of Feichtinger and Gr\"ochenig~\cite{fegr88,fegr89,fegr89-1,gr91}. Their results are heavily tied to the association of a time-frequency system to a group, e.g. the Heisenberg group and STFT or the affine group and wavelet transforms. More precisely, the time-frequency atoms are obtained through application of a square-integrable group representation to a prototype atom. There have been several attempts to loosen these restrictions to accommodate other group-related transforms, e.g. the 
 $\alpha$-transform~\cite{daforastte08} and and shearlet transform~\cite{kula12-alt,dakustte09}, see also the references given in~\cite{fora05}. Finally, the work of Fornasier and Rauhut~\cite{fora05} completely abolished the need for an underlying group in favor of general continuous frames that satisfy certain regularity conditions. Since our systems lack the relation to a group representation, the starting point of our investigation is a minor extension of their results.
  
 \subsection{Structure of this contribution}
   In the next section, we review some necessary theory and notation, including a short overview of the results presented in \cite{fora05}, the foundation on which the rest of this manuscript is built. In Section \ref{sec:genosc}, we introduce a minor but useful extension to these results that allows the treatment of the systems we wish to construct, but also the (intuitive) construction of Banach frames and atomic decompositions for the STFT using the Fornasier-Rauhut theory. The rest of the paper is focused on warped time-frequency representations, their definition and basic properties are presented in Section \ref{sec:warp}. Section \ref{sec:warpedcoorbits} is concerned with the construction of the associated test function spaces that in turn enable the definition of the corresponding coorbit spaces. Finally, Section \ref{sec:decomps} treats discretization results for warped time-frequency systems on the associated coorbit spaces.

\section{Preliminaries} 
   
  We use the following normalization of the Fourier transform
\begin{equation}
  \hat f(\xi) := \mathcal{F}f = \int_\RR f(t) e^{-2 \pi i t \xi} \; dt, \text 
{for all } f\in\LiR
\end{equation}
and its unitary extension to $\LtR$. The inverse Fourier transform is denoted by 
$\check{f} := \mathcal F^{-1}f$ Further, we require the \emph{modulation 
operator} and the \emph{translation operator} defined by 
$\bd{M}_\omega f = f\cdot e^{2\pi i\omega(\cdot)}$ and $\bd T_x f = f(\cdot -x)$ 
respectively for all $f \in \LtR$. 
The composition of two functions $f$ and $g$ is denoted by $f\circ g$. By a 
superscript asterisk ($^*$), we denote the adjoint of an operator and the 
anti-dual of a Banach space, i.e. the space of all continuous, conjugate-linear 
functionals on the space. The Landau notation $\mathcal O(f)$  
denotes all functions that do not grow faster than $f$.

Let $\mathcal H$ be a separable Hilbert space and $(X,\mu)$ a locally 
compact, $\sigma$-compact Hausdorff space with positive Radon measure $\mu$ on $X$. A nontrivial 
Banach space $(Y,\|\cdot\|_Y)$ of functions on $X$, continuously embedded in 
$\bd L^1_{loc}(X,\mu)$, is \emph{solid}, if for all $G\in \bd L^1_{loc}(X,\mu)$ 
and $H\in Y$
\begin{equation}
  |G(x)|\leq |H(x)| \text{ a.e. } \Rightarrow\ G\in Y \text{ and } \|G\|_Y \leq 
\|H\|_Y.
\end{equation}

A collection $\Psi = \{\psi_x\}_{x\in X}$ of functions $\psi_x\in\mathcal H$ is 
called a \emph{continuous frame}, if there are $0 < A\leq B <\infty$, such that
\begin{equation}
  A\|f\|_{\mathcal H}^2 \leq \int_X |\langle f,\psi_x\rangle|^2 d\mu(x) \leq 
B\|f\|_{\mathcal H}^2, \text{ for all } f\in\mathcal H,
\end{equation}
and the map $x\mapsto \psi_x$ is weakly continuous. A frame is called 
\emph{tight}, if $A=B$. For any frame, the \emph{frame operator} defined (in the 
weak sense) by
\begin{equation}
  \bd S_\Psi~:~ \mathcal H\rightarrow \mathcal H,\ \bd S_\Psi f := \int_X 
\langle f,\psi_x\rangle \psi_x d\mu(x),
\end{equation}
is bounded, positive and boundedly invertible~\cite{alanga93,rahim06}.

A \emph{kernel} on $X$ is a function $K:X\times X\rightarrow \CC$. Its 
application to a function $G$ on $X$ is denoted by 
\begin{equation}
  K(G)(x):= \int_X K(x,y)G(y)~d\mu(y).
\end{equation}

Although the theory in Sections \ref{ssec:coorbit}, \ref{ssec:discret} and 
\ref{sec:genosc} are valid in this general setting, the later sections mostly 
consider cases where $X$ is a suitable subset of $\RR^2$ endowed with the usual 
Lebesgue measure, and $\mathcal H$ is some subspace of $\mathbf{L}^2(\RR)$. 

The most important examples for $(Y,\|\cdot\|_Y)$ are weighted, mixed-norm spaces 
from the family $\bd 
L^{p,q}_w(X)$, for $1\leq p,q\leq \infty$, $X\subseteq \RR^{2d}$ and a continuous, nonnegative {weight 
function} $w:X\mapsto\RR$. These spaces consist of all Lebesgue measurable 
functions, such that the norm 
\begin{equation}
  \|G\|_{\bd L^{p,q}_w} := \left(\int_{\RR^d}\left(\int_{\RR^d} w(x,\xi)^p|G(x,\xi)|^p~ 
dx\right)^{q/p} ~d\xi\right)^{1/q} < \infty.
\end{equation}
Here, $G$ is identified with its zero-extension to a function on 
$\RR^{2d}$. If $p=\infty$ or $q=\infty$, the respective $p$-norm is replaced by the 
essential supremum as usual.

In the next subsections, we recall the central results of generalized coorbit 
theory and their requirements. The interested reader can find a more detailed 
discussion and the necessary proofs in~\cite{fora05}, where these results were 
first presented.

\subsection{The construction of generalized coorbit spaces}\label{ssec:coorbit}

   For the sake of brevity, we will assume from now on that $\Psi := 
\{\psi_x\}_{x\in X}\subset \mathcal H$ is a tight frame, i.e. $\bd S_\Psi f = 
Af$ for all $f\in \mathcal H$, leading to considerable simplifications in the 
following statements. Define the following transform associated to $\Psi$,
   \begin{equation} 
     V_\Psi~:~ \mathcal H \rightarrow L^2(X,\mu),\quad \text{defined by}\quad V_\Psi f(x):= \langle f, 
\psi_x \rangle.
   \end{equation}
   The adjoint operator is given in the weak sense by
   \begin{equation} 
     V_\Psi^\ast~:~ L^2(X,\mu) \rightarrow \mathcal H,\quad V_\Psi^\ast G := 
\int_X G(y)\psi_y d\mu(y).
   \end{equation}
   
   Furthermore, let $\mathcal A_1$ be the Banach algebra of all kernels 
$K:X\times X\rightarrow \CC$, such that the norm 
   \begin{equation}\label{eq:normA1}
      \|K\|_{\AAi} := \max\left\{\esssup_{x\in X} \int_X |K(x,y)|~d\mu(y),\ 
\esssup_{y\in X} \int_X |K(x,y)|~d\mu(x)\right\}
   \end{equation}
   is finite. Note that the two suprema are equal if $K$ is (Hermitian) 
symmetric. The algebra multiplication be given by
   \begin{equation}
      (K_1\cdot K_2)(x,y) = \int_X K_1(x,z)K_2(z,y)~dz.
   \end{equation}
   
   A weight function $m:X\times X\rightarrow \CC$, is called \emph{admissible} if
it satisfies
   \begin{equation}
     1\leq m(x,y) \leq m(x,z)m(z,y),\quad m(x,y) = m(y,x)\ \text{ and }\ m(x,x) 
\leq C,    
   \end{equation}
   for some $C>0$ and all $x,y,z\in X$. The weighted kernel algebra $\mathcal 
A_m$ is the space of all kernels $K:X\times X\rightarrow \CC$, such that 
   \begin{equation}
     \|K\|_{\AAm} := \|Km\|_{\AAi} < \infty.
   \end{equation}
   
   The following theorem combines several important results from~\cite{fora05}.

   \begin{theorem}\label{thm:resanddual}
     Let $m$ be an admissible weight function, fix $z\in X$ and define $v:= 
m(\cdot,z)$. If $\Psi\subset Y$ is a continuous tight frame and the kernel 
$K_\Psi:X\times X\rightarrow \CC$, given by 
     \begin{equation}\label{eq:framekern}
       K_\Psi(x,y) := A^{-1}\langle \psi_y,\psi_x \rangle \text{ for all } x,y\in X,
     \end{equation}
     is contained in $\AAm$, then
     \begin{equation}\label{eq:resspace}
       \mathcal H^1_v := \{f\in\mathcal H~:~ V_\Psi f\in \bd L^1_v\}, \text{ 
with the norm } \|f\|_{\mathcal H^1_v}:= \|V_\Psi f\|_{\bd L^1_v},
     \end{equation}
     is the \emph{minimal Banach space} containing all the frame elements 
$\psi_x$ and satisfying $\|f\|_B \leq Cv(x)$ for some $C>0$. Furthermore, 
$\mathcal H^1_v$ is independent of the particular choice of $z\in X$ and the 
expression $\|V_\Psi f\|_{\bd L^\infty_{1/v}}$ defines an equivalent norm on the 
anti-dual $(\mathcal H^1_v)^\ast$ of $\mathcal H^1_v$.
   \end{theorem}
   
   The result above enables the extension of $V_\Psi$ to the distribution space 
$(\mathcal H^1_v)^\ast$ by means of
   \begin{equation}
     V_\Psi f(x):= \langle f, \psi_x \rangle = f(\psi_x), \text{ for all } x\in 
X, f\in (\mathcal H^1_v)^\ast.
   \end{equation}
   $\mathcal H^1_v$ possesses a number of additional nice properties. For an 
exhaustive list, please refer to~\cite{fora05}. We only wish to note that 
$\mathcal H^1_v$ is dense and continuously embedded in $\mathcal H$, whereas 
$\mathcal H$ is weak-$\ast$ dense in $(\mathcal H^1_v)^\ast$, giving rise to a 
Banach-Gelfand triple~\cite{gevi64,an98-2,cofelu08}.
   
   If a solid Banach space $Y$ satisfies 
   \begin{equation}\label{eq:BScond}
     \AAm(Y)\subset Y \text{ and }\|K(F)\|_Y\leq \|K\|_{\AAm}\|F\|_Y, \text{ for 
all } K\in{\AAm},\ F\in Y,
   \end{equation}
 then we can define the coorbits of $Y$ with respect to the frame $\Psi$, 
provided $K_\Psi \in\AAm$.
   \begin{equation}\label{eq:coorbits}
     \text{Co}Y := \text{Co}(\Psi,Y) := \{f\in (\mathcal H^1_v)^\ast~:~ V_\Psi 
f\in Y\},
   \end{equation}
   with natural norms $\|f\|_{\text{Co}Y} := \|V_\Psi f\|_Y$.
    
   \begin{theorem}\label{thm:coorbits}
     Let $Y$ be a solid Banach space that satisfies Eq. \eqref{eq:BScond} and 
$K_\Psi(Y)\subset \bd L^\infty_{1/v}$, for some admissible weight $m$. If $\Psi$ 
is a continuous tight frame with $K_\Psi\in\AAm$, then 
$(\text{Co}Y,\|\cdot\|_{\text{Co}Y})$ is a Banach space.     
     For $G\in Y$, $G=K_\Psi(G) \Leftrightarrow G = V_\Psi f$ for some 
$f\in\text{Co}Y$. Furthermore,  the map $V:\text{Co}Y \rightarrow Y$ is an 
isometry on the closed subspace $K_\Psi(Y)$ of $Y$.     
   \end{theorem}
   
    In particular, 
    \begin{equation}
      \text{Co}\bd L^1_v = \mathcal{H}^1_v,\quad \text{Co}\bd L^\infty_{1/v} = 
(\mathcal{H}^1_v)^\ast\ \text{ and }\ \text{Co}\bd L^2 = \bd L^2.
    \end{equation} 
    
    The coorbit spaces $(\text{Co}Y,\|\cdot\|_{\text{Co}Y})$ are independent of 
the particular choice of the continuous frame $\Psi$, under a certain 
equivalence condition on the mixed kernel associated to a pair of continuous 
frames. 
    
    \begin{proposition}\label{pro:mixedkern}
      If $Y$ and $\Psi$ satisfy the conditions of Theorem \ref{thm:coorbits} and 
$\widetilde{\Psi}$ is a continuous frame with 
$K_{\widetilde{\Psi}},K_{\Psi,\widetilde{\Psi}}\in\AAm$, where 
$K_{\Psi,\widetilde{\Psi}}$ is the mixed kernel defined by
      \begin{equation}
        K_{\Psi,\widetilde{\Psi}}(x,y) := \left\langle 
\widetilde{\psi_y},\psi_x\right\rangle  
      \end{equation}      
      then 
      \begin{equation}
        \text{Co}(\Psi,Y) = \text{Co}(\widetilde{\Psi},Y).
      \end{equation}
    \end{proposition}
    
\subsection{Discretization in generalized coorbit spaces}\label{ssec:discret}

    In the next steps, we investigate the discretization properties of the 
continuous frame $\Psi$, obtaining sufficient conditions for atomic 
decompositions and Banach frames in terms of a discrete subset of $\Psi$.

We only provide a review of the theory provided in \cite{fora05}, shortened to an absolute
minimum. For a comprehensive treatment including the necessary details, please refer to the original contribution.

    \begin{definition}\label{def:modadmcover}
      A family $\mathcal U = \{U_i\}_{i\in I}$ for some countable index set $I$ 
is called \emph{admissible covering} of $X$, if the following hold. Every $U_i$ 
is relatively compact with non-void interior, $X = \cup_{i\in I}$ and 
$\sup_{i\in I} \#\{j\in I~:~ U_i\cap U_j\neq\emptyset\} \leq N < \infty$ for 
some $N>0$. An admissible covering is \emph{moderate}, if $0< D\leq \mu(U_i)$ 
for all $i\in I$ and there is a constant $\tilde{C}$ with 
      \begin{equation}
        \mu(U_i)\leq \tilde{C}\mu(U_j), \text{ for all } i,j\in I \text{ such 
that } U_i\cap U_j\neq\emptyset.
      \end{equation}
    \end{definition}
    
%
    The main discretization result states that any pair of a continuous tight 
frame $\Psi$ and a covering $\mathcal U^\delta$, such that 
$\|\text{osc}_{\mathcal U^\delta}\|_{\AAm} < \delta$ for a sufficiently small 
$\delta > 0$, gives rise to atomic decompositions and Banach frames for $ 
\text{Co}(\Psi,Y)$ in a natural way, i.e. $\{\psi_{x_i}\}_{i\in I}$ is both a 
Banach frame and an atomic decomposition if $x_i\in U_i$ for all $i\in I$.
    
    \begin{definition}\label{def:atomdecbanfram}
      A family $\Psi:= \{\psi_i\}_{i\in I}$ in a Banach space $(B,\|\cdot\|_B)$ 
is called an \emph{atomic decomposition} for $B$, if there is a 
BK-space\footnote{A solid Banach space of sequences where convergence implies 
componentwise convergence.} $(B^\sharp,\|\cdot\|_B^\sharp)$ and linear, bounded 
functionals $\{\lambda_i\}_{i\in I} \subseteq B^\ast$ such that 
      \begin{itemize}
       \item $(\lambda_i(f))_{i\in I}\in B^\sharp$ for all $f\in B$ and there is 
a finite constant $C_1>0$ such that 
	  \begin{equation}
	    \|(\lambda_i(f))_{i\in I}\|_{B^\sharp} \leq C_1\|f\|_B,
	  \end{equation}
       \item if $(\lambda_i)_{i\in I}\in B^\sharp$, then $f := \sum_{i\in I} 
\lambda_i g_i \in B$ (with unconditional convergence in some suitable topology) 
and there is a finite constant $C_2 > 0$ such that
	  \begin{equation}
	    \|f\|_B \leq C_2\|(\lambda_i)_{i\in I}\|_{B^\sharp},
	  \end{equation}
       \item $f = \sum_{i\in I} \lambda_i(f) g_i$, for all $f\in B$.
      \end{itemize}
      A family $\widetilde{\Psi}:= \{\widetilde{\psi_i}\}_{i\in I}$ in 
$B^\ast$ is \emph{Banach frame} for $B$, if there is a BK-space 
$(B^\flat,\|\cdot\|_B^\flat)$ and linear, bounded operator $\Omega:B^\flat 
\rightarrow B$ such that 
      \begin{itemize}
       \item if $f\in B$, then $(\widetilde{\psi_i}(f))_{i\in I}\in B^\flat$ and 
there are finite constants $0< C_1 \leq C_2$ such that 
	  \begin{equation}
	    C_1\|f\|_B \leq \|(\widetilde{\psi_i}(f))_{i\in I}\|_{B^\flat} \leq 
C_2\|f\|_B,
	  \end{equation}
       \item $f = \Omega\left((\widetilde{\psi_i}(f))_{i\in I}\right)$, for all 
$f\in B$.
      \end{itemize}
    \end{definition}
    
    \begin{definition}\label{def:osckern}
    The oscillation of a continuous tight frame $\Psi$ with respect to 
the moderate, admissible covering $\mathcal U$ of $X$ is defined by
    \begin{equation}
      \text{osc}_{\mathcal U}(x,y):= \text{osc}_{\Psi,\mathcal U}(x,y):= 
A^{-1}\sup_{z\in Q_y} |\langle \psi_x,\psi_y-\psi_z \rangle| = A^{-1}\sup_{z\in Q_y} 
|K_\Psi(x,y)-K_\Psi(x,z)|,
    \end{equation}
    where $Q_y:= Q_{\mathcal U,y}:=\cup_{i\in I,y\in U_i} U_i$.
    \end{definition}
    
    \begin{theorem}\label{thm:discret1}
      Let $Y$ and $\Psi$ satisfy the conditions of Theorem \ref{thm:coorbits}. 
If the moderate, admissible covering $\mathcal U^\delta$ is such that 
      \begin{equation}
        \|\text{osc}_{\mathcal 
U^\delta}\|_{\AAm}\left(\|K_\Psi\|_{\AAm}+\max\{C_{m,\mathcal U^\delta} 
\|K_\Psi\|_{\AAm},\|K_\Psi\|_{\AAm}+\|\text{osc}_{\mathcal 
U^\delta}\|_{\AAm}\}\right)< 1,
      \end{equation}
      for some $C_{m,\mathcal U^\delta} \geq \sup_{i\in I}\sup_{x,y\in 
U_i^\delta} m(x,y)$, then $\{\psi_{x_i}\}_{i\in I}$ is a Banach frame and an 
atomic decomposition for $ \text{Co}(\Psi,Y)$ if $x_i\in U_i$ for all $i\in I$.
    \end{theorem}
    
    For details, e.g. about suitable associated sequence spaces, please refer to \cite{fora05}.
    
    The usual strategy for satisfying Theorem \ref{thm:discret1} is to construct 
a family of moderate, admissible coverings $\mathcal U^\delta$, such that 
    \begin{equation}\label{eq:convergentcover}
      \|\text{osc}_{\mathcal U^\delta}\|_{\AAm} \overset{\delta \rightarrow 
0}{\rightarrow} 0\text{ and } C_{m,\mathcal U^\delta} < C < \infty,
    \end{equation}
    for $\delta$ sufficiently small. Consequently, we can find $\delta_0 > 0$, 
such that Theorem \ref{thm:discret1} holds for all $\mathcal U^\delta$ with 
$\delta \leq \delta_0$.

\section{The generalized oscillation kernel}\label{sec:genosc}

We now motivate and present a generalization of the discretization theory for generalized coorbit 
spaces. Since the derivation of our extended results is largely analogous to the content of \cite[Section 5]{fora05},
we only provide the results and indicate the necessary changes here. However,
the complete derivation can be found in \cite{ournote}. There, we provide a variant of 
 \cite[Section 5]{fora05} considering our changes, as well as some corrections and modifications to 
 provide a more rigorous and accessible treatment of the theory provided in \cite{fora05}.

A closer investigation of the oscillation kernel associated to the short-time 
Fourier transform (STFT) shows that the sampling results obtained via classical 
coorbit space theory are not easily recovered using the 
theory presented in~\cite{fora05}. At the very least, no sequence of intuitive, 
\emph{regular} phase space coverings with property \eqref{eq:convergentcover} 
seems to exist, as Example \ref{ex:STFTcover} below demonstrates. We conclude 
that the construction of a moderate, admissible covering with 
$\|\text{osc}_{\mathcal U^\delta}\|_{\AAm} < \delta$ is far from a trivial task, 
if at all possible. In the setting of the $\alpha$-transform, Dahlke et 
al~\cite{daforastte08} circumvent this problem by redefining the oscillation 
kernel to take into account the group action on the \emph{affine Weyl-Heisenberg 
group}. An alternative approach, obtaining semi-regular Banach frames from 
sampled $\alpha$-transforms, is presented in \cite{fo07}, which is in turn based 
on previous work by Feichtinger and Gr\"ochenig~\cite{fegr92}. The following 
examples serve to illustrate why the (unaltered) application of generalized 
coorbit theory to the STFT and $\alpha$-transform presents a nontrivial task. At 
the same time, they 
motivate our own solution to the problem.

\begin{example}[Coverings for the STFT]\label{ex:STFTcover}
  Define the covering $\mathcal U^\delta := \{U_{k,l}^\delta\}_{k,l\in\ZZ}$ 
by 
  \begin{equation}
    U_{0,0}^{\delta}:= (-\delta,\delta)\times (-\delta,\delta),\quad 
U_{k,l}^\delta := \bd T_{(k\delta,l\delta)}U_{0,0}^{\delta}, \text{ for all } 
k,l\in\ZZ.
  \end{equation}
  Selecting the Schr\"{o}dinger representation of the Heisenberg group, the 
continuous tight frame of short-time Fourier type arising from the Schwartz 
class window $g\in\mathcal S(\RR)$ is given by $\mathcal 
G(g):=\{g_{x,\xi}\}_{x,\xi\in\RR}$, where
  \begin{equation}
    g_{x,\xi} := e^{-\pi ix\xi}\bd M_\xi \bd T_x g.
  \end{equation}
  In classical coorbit theory, see e.g. \cite{gr91}, the associated oscillation 
kernel with respect to $\mathcal U^\delta$ is given by
  \begin{equation}
    \begin{split}
    \widetilde{\text{osc}}_{\mathcal U^\delta}(x,y,\xi,\omega) & := 
\sup_{(z,\eta)\in Q_{(y,\omega)}} |V_{G(g)} g(y-x,\omega-\xi) - V_{G(g)} 
g(z-x,\eta-\xi)| \\
    & \leq \sup_{(\epsilon_1,\epsilon_2)\in U^{2\delta}_{0,0}} |V_{G(g)} 
g(y-x,\omega-\xi) - V_{G(g)} g(y+\epsilon_1-x,\omega+\epsilon_2-\xi)| \\
    & = \widetilde{\text{osc}}_{\mathcal U^{2\delta}}(0,y-x,0,\omega-\xi),
    \end{split}
  \end{equation}
  since $Q_{y,\omega} \subseteq U_{y,\omega}^{2\delta} := 
(y-2\delta,y+2\delta)\times (\omega-2\delta,\omega+2\delta)$. Now, 
  \begin{equation}
    V_{G(g)} g(y,\omega) - V_{G(g)} g(y+\epsilon_1,\omega+\epsilon_2) = 
\left\langle g, e^{-\pi i y\omega}\bd M_\omega \bd T_y \left( g - e^{\pi 
i(y\epsilon_2-\omega\epsilon_1)}\bd M_{\epsilon_2} \bd T_{\epsilon_1} g \right) 
\right\rangle.
  \end{equation}
Note that $e^{\pi i(y\epsilon_2-\omega\epsilon_1)} \overset{ 
(\epsilon_1,\epsilon_2)\rightarrow (0,0)}{\rightarrow} 1$ for any fixed 
$(y,\omega)\in\RR^2$ and $\langle g,g_{y,\omega} \rangle$ rapidly converges to 
$0$, for $|(y,\omega)|\rightarrow \infty$. Therefore, a standard $2\epsilon$ 
argument, considering $\widetilde{\text{osc}}_{\mathcal 
U^{2\delta}}(0,y,0,\omega)$ on a compact neighborhood of $(0,0)$ and 
$\sup_{(z,\eta)\in Q_{(y,\omega))}} 2|\langle g,g_{z,\eta} \rangle|$ shows that 
Eq. \eqref{eq:convergentcover} holds.

On the other hand, the oscillation kernel according to Definition 
\ref{def:osckern} and \cite{fora05} yields
\begin{equation}
    \begin{split}
    \text{osc}_{\mathcal U^\delta}(x,y,\xi,\omega) & = \sup_{(z,\eta)\in 
Q_{(y,\omega)}} |e^{\pi i (y\xi-x\omega)}V_{G(g)} g(y-x,\omega-\xi) - e^{\pi i 
(z\xi-x\eta)}V_{G(g)} g(z-x,\eta-\xi)| \\
    & = \sup_{(z,\eta)\in Q_{(y,\omega)}} |V_{G(g)} g(y-x,\omega-\xi) - e^{\pi i 
\left((z-y)\xi-x(\eta-\omega)\right)}V_{G(g)} g(z-x,\eta-\xi)|, 
    \end{split}
  \end{equation}
  and 
  \begin{equation}
  \begin{split}
   \lefteqn{V_{G(g)} g(y-x,\omega-\xi) - e^{\pi i 
\left((z-y)\xi-x(\eta-\omega)\right)}V_{G(g)} 
g(y-x+\epsilon_1,\omega-\xi+\epsilon_2)}\\
   & = \left\langle g, e^{-\pi i (y-x)(\xi-\omega)}\bd M_{\omega-\xi} \bd 
T_{y-x} \left( g - e^{\pi i (y\epsilon_2-\epsilon_1\omega)}\bd M_{\epsilon_2} 
\bd T_{\epsilon_1} g \right) \right\rangle.
  \end{split}
  \end{equation}
  Note how the inner phase factor $e^{\pi i (y\epsilon_2-\epsilon_1\omega)}$ does not 
depend on $|x-y|$ or $|\xi-\omega|$. It is easy to see that we can find, for 
every $\epsilon_1,\epsilon_2 > 0$, a pair $(y,\omega)\in\RR^2$, such that 
$e^{\pi i (y\epsilon_2-\epsilon_1\omega)} = -1$. By a similar $2\epsilon$ 
argument to before, we can construct a sequence 
$\{(y_j,\omega_j,\epsilon_j)\}_{j\in\NN}$, such that $\epsilon_j\rightarrow 0$ 
as $j\rightarrow \infty$ and 
\begin{equation}
    \begin{split}
   \esssup_{x,\xi\in\RR} \int_\RR \int_\RR \left|\left\langle g, \bd 
M_{\omega_j-\xi} \bd T_{y_j-x} \left( g - e^{\pi i 
(y_j\epsilon_j-\epsilon_j\omega_j)}\bd M_{\epsilon_j} \bd T_{\epsilon_j} g 
\right) \right\rangle\right|~d\omega~dy \overset{j \rightarrow \infty} 
\rightarrow \int_\RR \int_\RR \left|\left\langle g, 2\bd M_{\omega} \bd T_{y} 
g\right\rangle\right|~d\omega~dy.
  \end{split}
  \end{equation}
  Similarly,
  \begin{equation}
    \begin{split}
   \esssup_{y,\omega\in\RR} \int_\RR \int_\RR \left|\left\langle g, \bd 
M_{\omega-\xi} \bd T_{y-x} \left( g - e^{\pi i (y\epsilon-\epsilon\omega)}\bd 
M_{\epsilon} \bd T_{\epsilon} g \right) \right\rangle\right|~d\xi~dx 
\overset{\epsilon \rightarrow 0} \rightarrow \int_\RR \int_\RR 
\left|\left\langle g, 2\bd M_{\xi} \bd T_{x} g\right\rangle\right|~d\xi~dx.
  \end{split}
  \end{equation}
  Consequently, $\|\text{osc}_{\mathcal U^\delta}\|_{\AAm} \overset{\delta 
\rightarrow 0} \rightarrow 2\|K_{\mathcal{G}(g)}\|_{\AAm}$ and the family 
$\mathcal U^\delta$ does not satisfy Eq. \eqref{eq:convergentcover}. Neither 
does any family of coverings constructed from regular phase space shifts of 
a fixed compact set $U\subset \RR^2$. A similar 
argument provides the same result for any other sensible definition of the 
STFT. 
\end{example}

While we cannot prove that there is no family of moderate, admissible coverings 
with the property Eq. \eqref{eq:convergentcover}, it is surely much harder to 
satisfy using Definition \ref{def:osckern}, than using classical 
theory~\cite{gr91}. A similar situation arises for the so-called 
$\alpha$-transform~\cite{daforastte08,fefo06}. However, in that situation, 
classical coorbit theory does not apply and we must rely on its generalized 
variant.

\begin{example}[The oscillation kernel for the 
$\alpha$-transform]\label{ex:ALPHAcover}
  For $\alpha\in [0,1[$ and a function $g\in\mathcal S(\RR)$, let $\mathcal 
G_\alpha(g):=\{g_{x,\xi}\}_{x,\xi\in\RR}$, with 
  \begin{equation}
    g_{x,\xi} := \bd T_x \bd M_\xi \bd D_{\beta_\alpha(\xi)} g,
  \end{equation}
  where $\beta_\alpha(\xi) := (1+|\xi|)^{-\alpha}$ and $\bd D_{a}g := a^{-1/2} 
g(\cdot/a)$ is the unitary dilation. Then 
  \begin{equation}
    g_{y,\omega}-g_{z,\eta} = \bd T_x \bd M_\xi \bd D_{\beta_\alpha(\xi)} (g - 
e^{-2\pi i\omega (z-y)}\bd T_{\beta_\alpha(\omega)^{-1}(z-y)}\bd 
M_{\beta_\alpha(\omega)(\eta-\omega)}\bd 
D_{\beta_\alpha(\eta)/\beta_\alpha(\omega)} g).
  \end{equation}  
  This suggests the construction of a moderate admissible covering from a 
countable subset of $\{ U^{\delta}_{x,\xi}\}_{x,\xi\in \RR}$, 
$U^{\delta}_{x,\xi} := 
(x-\beta_\alpha(\xi)\delta,x+\beta_\alpha(\xi)\delta)\times 
(\xi-\beta_\alpha(\xi)^{-1}\delta,\xi+\beta_\alpha(\xi)^{-1}\delta)$. Hence, 
$(z,\eta)\in Q_{y,\omega}$ implies
  \begin{equation}
    (z-y) \sim \beta_\alpha(\omega)\delta.
  \end{equation}
  Although $e^{-2\pi i\omega\delta\beta_\alpha(\omega)} = e^{-2\pi 
i\delta\omega(1+|\omega|)^{-\alpha}}$ converges to $1$ for $\delta\rightarrow 
0$, convergence speed decreases in $|\omega|$, for all $0\leq \alpha < 1$. 
Similar to Example \ref{ex:STFTcover}, the phase factor can behave arbitrarily 
bad, independent of the size of the covering elements. In \cite{daforastte08},
this problem is circumvented by redefining the oscillation to respect the group action.
\end{example}

The negative results obtained in the examples above motivate a more general 
definition of the oscillation kernel. With the following extended definition, 
the construction of a covering family with the property Eq. 
\eqref{eq:convergentcover} becomes a properly intuitive task, similar to the 
classical case~\cite{gr91}.

\begin{definition}[Definition \ref{def:osckern}a]\label{def:genosckern}
    The $\Gamma$-oscillation of a continuous tight frame $\Psi$ with 
respect to the moderate, admissible covering $\mathcal U$ of $X$ is defined by
    \begin{equation}
      \text{osc}_{\mathcal U,\Gamma}(x,y):= A^{-1}\sup_{z\in Q_y} |\langle 
\psi_x,\psi_y-\Gamma(y,z)\psi_z \rangle| = A^{-1}\sup_{z\in Q_y} 
|K_\Psi(x,y)-\Gamma(y,z)K_\Psi(x,z)|,
    \end{equation}
    where $Q_y:=\cup_{i\in I,y\in U_i} U_i$ and $\Gamma:X\times X \rightarrow 
\CC$ satisfies $|\Gamma|=1$.
\end{definition}

At first glance, the above definition might seem arbitrary, but it actually gives rise 
to a simple generalization of Theorem \ref{thm:discret1}.

\begin{theorem}[Theorem \ref{thm:discret1}a]\label{thm:discret2}
      Let $Y$ and $\Psi$ satisfy the conditions of Theorem \ref{thm:coorbits}. 
If there is some $\Gamma:X\times X \rightarrow \CC$, $|\Gamma|=1$, and a 
moderate, admissible covering $\mathcal U^\delta$ such that 
      \begin{equation}
        \|\text{osc}_{\mathcal 
U^\delta,\Gamma}\|_{\AAm}\left(\|K_\Psi\|_{\AAm}+\max\{C_{m,\mathcal U^\delta} 
\|K_\Psi\|_{\AAm},\|K_\Psi\|_{\AAm}+\|\text{osc}_{\mathcal 
U^\delta,\Gamma}\|_{\AAm}\}\right)< 1,
      \end{equation}
      for some $C_{m,\mathcal U^\delta} \geq \sup_{i\in I}\sup_{x,y\in 
U_i^\delta} m(x,y)$, then $\{\psi_{x_i}\}_{i\in I}$ is a Banach frame and an 
atomic decomposition for $ \text{Co}(\Psi,Y)$ if $x_i\in U_i$ for all $i\in I$.
\end{theorem}

\begin{remark}
  The result above is truly different from Theorem \ref{thm:discret1}, only 
when $\Gamma$ is not separable into two independent phase factors of the same 
form, i.e. 
  \begin{equation}
    \nexists\ \Gamma_1:X \rightarrow \CC \text{ with } |\Gamma_1|=1, \text{ such 
that } \Gamma(y,z) = \Gamma_1(y)^{-1}\Gamma_1(z), \text{ for all } y,z\in X.
  \end{equation}
  Otherwise, $\widetilde{\psi_x}:= \Gamma_1(x)\psi_x$ defines a continuous frame 
that provides essentially the same transform, gives rise to the same coorbit 
spaces and satisfies Theorem \ref{thm:discret1}.
\end{remark}

Proving Theorem \ref{thm:discret1} is is a lengthy affair, see \cite{fora05}, 
and requires a substantial number of interim results, most of which do not even 
reference the oscillation kernel. 
All this preparation can be done in exactly the same way to prove Theorem 
\ref{thm:discret2}. 
To be precise, the oscillation kernel appears only in the proofs for Lemmas 
5.5, 5.10 and 5.11, as well as Theorem 5.13 in \cite{fora05}. 
Moreover, the proofs of Lemmas 5.5, 5.10 and 5.11 can be executed identically 
for the generalized oscillation kernel from Definition \ref{def:genosckern}, 
requiring only $|K_\Psi(x,y)| = |\Gamma(x,y)K_\Psi(x,y)|$.

The crucial step for proving Theorem \ref{thm:discret2}, however, is the 
invertibility of the discretization operator $U_{\Psi}$, defined by
\begin{equation}
   \bd U_{\Psi}(G)(x):= \sum_{i\in I} c_i G(x_i)K_{\Psi}(x,x_i), \text{ for all 
} G\in Y
\end{equation}
where
\begin{equation}
   c_i := \int_{X} \phi_i(y)~ d\mu(y)
\end{equation}
and $\Phi:=\{\phi_i\}_{i\in I}$ is a partition of unity with respect to the 
moderate, admissible covering $\mathcal U^\delta := \{U_i^\delta\}_{i\in I}$, 
i.e. 
\begin{equation}
  \sum_{i\in I} \phi_i = 1\quad \text{ and } \quad \supp(\phi_i)\subseteq 
U_i^\delta.
\end{equation}

This is achieved by the following theorem, a variant of \cite[Theorem 5.13]{fora05}.

\begin{theorem}
  Let $Y$ and $\Psi$ satisfy the conditions of Theorem \ref{thm:coorbits}. If 
there is some $\Gamma:X\times X \rightarrow \CC$, $|\Gamma|=1$, and a moderate, 
admissible covering $\mathcal U^\delta$ such that 
  $\|\text{osc}_{\mathcal U^\delta,\Gamma}\|_{\AAm} < \delta$, then 
  \begin{equation}\label{eq:IDminusU}
       \|\operatorname{Id}-\bd U_{\Psi}\|_{K_{\Psi}(Y)\rightarrow K_{\Psi}(Y)} 
\leq \delta\left(\|K_\Psi\|_{\AAm}+\max\{C_{m,\mathcal U^\delta} 
\|K_\Psi\|_{\AAm},\|K_\Psi\|_{\AAm}+\delta\}\right).
  \end{equation}
  In particular, $\bd U_{\Psi}$ is bounded and if the RHS of Eq. 
\eqref{eq:IDminusU} is less or equal to $1$, then $\bd U_{\Psi}$ is boundedly 
invertible on $K_{\Psi}(Y)$
\end{theorem}
\begin{proof}
  For the assertion $\bd U_{\Psi}G\in K_{\Psi}(Y)$, please refer to 
\cite{fora05}. To prove the norm estimate, we introduce the auxiliary operator
      \begin{equation}
      \begin{split}
        \bd S_{\Psi} G(x) & := K_{\Psi}\left(\sum_{i\in I} 
G(x_i)\overline{\Gamma(\cdot,x_i)}\phi_i\right)(x).
      \end{split}
      \end{equation}
      Note that $K_{\Psi}$ equals the identity on $K_{\Psi}(Y)$, by Theorem 
\ref{thm:coorbits}; to confirm $K_{\Psi}(Y)\subseteq \bd L^\infty_{1/v}$ with 
continuous embedding, refer to \cite[Corollary 5.6]{fora05}. By the triangle 
inequality, 
      \begin{equation}
        \|G-\bd U_{\Psi}G\|_Y \leq \|G-\bd S_{\Psi}G\|_Y + \|\bd S_{\Psi}G-\bd U_{\Psi}G\|_Y.
      \end{equation}
      We now estimate both terms on the RHS separately.
      \begin{equation}
      \begin{split}
        \|G-\bd S_{\Psi} G\|_Y & = \|K_{\Psi} G - \bd S_{\Psi} G\|_Y \\ 
        & \leq \|K_{\Psi}\|_{\AAm}\|\sum_{i\in 
I}\left(G-G(x_i)\overline{\Gamma(\cdot,x_i)}\right)\phi_i\|_Y.
      \end{split}
      \end{equation}
      In order to estimate $\|\left(G-G(x_i)\overline{\Gamma(\cdot,x_i)}\right)\phi_i\|_Y$, 
examine
      \begin{equation}
      \begin{split}
        \lefteqn{\left|\sum_{i\in I} 
\left(G(x)-\overline{\Gamma(y,x_i)}G(x_i)\right)\phi_i(x)\right| = \left|\sum_{i\in I} 
\left(K_{\Psi}(G)(x)-\overline{\Gamma(x,x_i)}K_{\Psi}(G)(x_i)\right)\phi_i(x)\right|}\\
        &= \left|\sum_{i\in I} \int_{X} 
G(y)\left(K_{\Psi}(x,y)-\overline{\Gamma(x,x_i)}K_{\Psi}(x_i,y)\right)~d\mu(y)\phi_i(x)\right|\\
&\leq \sum_{i\in I} \int_{X} 
G(y)\left|K_{\Psi}(y,x)-\overline{\Gamma(x,x_i)}K_{\Psi}(y,x_i)\right|~d\mu(y)\phi_i(x)\\
        &\leq \sum_{i\in I} \int_{X} |G(y)| \text{osc}_{\mathcal 
U^\delta,\Gamma}(y,x)~d\mu(y)\phi_i(x)\\   
        &= \sum_{i\in I} \text{osc}_{\mathcal 
U^\delta,\Gamma}^\ast(|G|)(x)\phi_i(x) = \text{osc}_{\mathcal 
U^\delta,\Gamma}^\ast(|G|)(x).
      \end{split}
      \end{equation}
      In the derivations above, we used $K_{\Psi}(x,y) = \overline{K_{\Psi}(y,x)}$ and the property $\supp(\phi_i)\subseteq U_i^\delta \in \mathcal{U}$ of the PU $\Phi=(\phi_i)_{i\in I}$.
      We obtain 
      \begin{equation}
        \|G-\bd S_{\Psi} G\|_Y \leq \|K_{\Psi}\|_{\AAm}\|\text{osc}_{\mathcal 
U^\delta,\Gamma}\|_{\AAm}\|G\|_Y,
      \end{equation}
      since $\|\oscUGd^\ast\|_{\AAm} = \|\oscUGd\|_{\AAm}$. 
      
      Now, we estimate $\|\bd S_{\Psi}G-\bd U_{\Psi}G\|_Y$. Note that
      \begin{equation}
       \begin{split}
        \lefteqn{|\bd S_{\Psi}(G)(x) - \bd U_{\Psi}(G)(x)|}\\
        & = \left|\sum_{i\in I} \int_X G(x_i)\phi_i(y) 
\left(\overline{\Gamma(y,x_i)}K_{\Psi}(x,y)-K_{\Psi}(x,x_i)\right)~d\mu(y)\right|\\
        & \leq \sum_{i\in I} \int_X |G(x_i)|\phi_i(y) 
\left|K_{\Psi}(x,y)-\Gamma(y,x_i)K_{\Psi}(x,x_i)\right|~d\mu(y)\\
        & \leq \sum_{i\in I} \int_X |G(x_i)|\phi_i(y)\oscUGd(x,y)~d\mu(y),
       \end{split}       
      \end{equation}
      where we used $\supp(\phi_i)\subseteq U_i \in \mathcal{U}^\delta$ once more.
      
      Define $H(y):= \sum_{i\in I} |G(x_i)|\phi_i(y)$, then by \cite[Lemma 
5.11]{fora05} and solidity of $Y$:
      \begin{equation}
      \begin{split}
	\|\bd S_{\Psi}G-\bd U_{\Psi}G\|_Y & \leq \|\text{osc}_{\mathcal 
U^\delta,\Gamma}\|_{\AAm}\|H\|_Y\\
	& \leq \max\{C_{m,\mathcal U^\delta} 
\|K_\Psi\|_{\AAm},\|K_\Psi\|_{\AAm}+\|\text{osc}_{\mathcal 
U^\delta,\Gamma}\|_{\AAm}\}\|\text{osc}_{\mathcal 
U^\delta,\Gamma}\|_{\AAm}\|G\|_Y.
      \end{split}
      \end{equation}
      Insert $\|\text{osc}_{\mathcal U^\delta,\Gamma}\|_{\AAm}<\delta$ to 
complete the proof.      
\end{proof}

With the result above in place and the changes discussed earlier in this 
section, the proof of suitable variants of \cite[Theorems 5.7, 5.8]{fora05} 
using $\text{osc}_{\mathcal U^\delta,\Gamma}$ is identical to the one presented 
by Fornasier and Rauhut~\cite{fora05}. 

Our own result in Theorem \ref{thm:discret2} is weaker than these variants of \cite[Theorems 5.7, 5.8]{fora05}
and therefore implied.

This concludes our discussion of abstract coorbit and discretization theory, note again 
that a fully fledged variant of \cite[Section 5]{fora05}, adjusted to the $\Gamma$-oscillation
can be found in \cite{ournote}. In 
the following sections, we will construct a family of time-frequency 
representations and apply the results obtained so far in their context.

\section{Warped time-frequency representations}\label{sec:warp}

In this section time-frequency representations with uniform frequency resolution on nonlinear frequency scales are constructed and their basic properties are investigated. In particular, we show that these transforms are continuous, norm preserving and invertible. 

Our method, motivated by the discrete systems in \cite{howi14}, is based on the simple premise of a function system $(\psi_{x,\xi})_{(x,\xi)\in D\times\RR}$, such that $\psi_{x,\xi} = \bd T_\xi\psi_{x,0}$, where $\psi_{x,0}$ and $\psi_{y,0}$ are of identical shape when observed on the desired frequency scale, for all $x,y\in\RR$. The frequency scale itself is determined by the so-called \emph{warping function}.
Generally, any bijective, continuous and increasing function $F:D\mapsto \RR$, where $D$ is an interval, specifies a (frequency) scale on $D$. More explicitly, for a prototype function $\theta:\RR\mapsto\CC$ and warping function $F$, the time-frequency atoms are given by
\begin{equation}\label{eq:thetafm}
 g_{x,\xi} := \bd T_\xi \mathcal F^{-1} \theta_{F,x}, \, \text{where } \theta_{F,x} = \bd (T_{F(x)}\theta)\circ F,
\end{equation}
up to a suitable normalization factor, see below. For the sake of simplicity, we consider here only the two most important cases $D=\RR$ or $D=\RR^+$.

This method allows for a large amount of flexibility when selecting the desired frequency scale, but we also recover classical time-frequency and time-scale systems:  
Clearly, we obtain a regular system of translates for any linear function $F$, while observing $(T_x\theta)\circ \log_a = (\theta\circ \log_a)(\cdot/a^x)$ shows that logarithmic $F$ provides a system of dilates, respectively. Therefore, short-time Fourier~\cite{ga46,gr01,fest98,fest03} and wavelet~\cite{ma09-1,da92} transforms will turn out to be special cases of our setting. 
In order to obtain \emph{nice} systems, we require the derivative of the inverse warping function $(F^{-1})'$ to be a $v$-moderate weight function.

\begin{definition}
  \begin{itemize}
    \item A weight function $v: \RR \rightarrow \RR^+$ is called 
\emph{submultiplicative} if
      \begin{equation}
	v(x+y)\leq v(x)v(y).
	\label{}
      \end{equation}
    \item A weight function $w: \RR \rightarrow \RR^+$ is called 
\emph{$v$-moderate} if
      \begin{equation}\label{eq:moderateness}
	w(x+y) \leq C v(x) w(y),
      \end{equation}
      for some submultiplicative weight function $v$ and constant $C<\infty$.
  \end{itemize}
\end{definition}

Submultiplicative and moderate weight functions are an important concept in the 
theory of function spaces, as they are closely related to the translation 
invariance of the corresponding weighted spaces \cite{fe79-2,gr01}, see 
also~\cite{gr07} for an in depth analysis of weight functions and their role in 
harmonic analysis.

\begin{definition}\label{def:warpfun}
  Let $D\in\{\RR,\RR^+\}$. A bijective function $F: D \rightarrow \RR$ is called \emph{warping function}, if $F \in \mathcal C^1(D)$ with $F'>0$, $|t_0|<|t_1| \Rightarrow F'(t_1)<F'(t_0)$ and the associated weight function
  \begin{equation}
    w(t) = \left( F^{-1} \right)'(t) = \frac{1}{F'\left(F^{-1}(t) \right)},
    \label{}
  \end{equation}
  is $v$-moderate for some submultiplicative weight $v$. If $D=\RR$, we additionally require $F$ to be odd.
\end{definition}

\begin{remark}\label{rem:transinv}
  Moderateness of $w = \left( F^{-1} \right)'$ ensures translation invariance of the associated \emph{weighted $\bd L^p$ spaces}. In particular, 
  \begin{equation}\label{eq:moderatenormbound}
    \norm{(T_x\theta)\circ F}^2_{\LtD} = \norm{T_x \theta}^2_{\bd L^2_{\sqrt{w}}(\RR)} \leq Cv(F(x)) \norm{\theta}^2_{\bd L^2_{\sqrt{w}}(\RR)}
  \end{equation}
  holds for all $\theta \in \bd L^2_{\sqrt{w}}(\RR)$. As usual, $\mathbf{L}_w^2(\RR)$ is the space of functions $f~:~\RR\mapsto\CC$, such that $wf\in \bd L^2(\RR)$.
\end{remark}

\begin{remark}
  The definition above only allows warping functions with nonincreasing derivative and, if $D=\RR$, we also require point-symmetry. Both constraints can be easily weakened by extending the warping function definition to any function $F$, such that there exists another function $F_0$ satisfying $F^{-1}(0) = F_0^{-1}(0)$ and $F'(x)\geq F_0'(x)$, with $F_0$ being a warping function as per Definition \ref{def:warpfun}. The results in this section still hold for such warping functions $F$. 
  
  Similar to the results in \cite{howi14}, it might be possible to connect the coorbit spaces and discretization results with respect to $F_0$ to the representation with respect to $F$. A thorough investigation of this issue is beyond the scope of this paper, however. Since all the examples that we wish to discuss have nonincreasing derivative in the first place, we will omit the details of working with two distinct warping functions $F$, $F_0$.
\end{remark}

From here on, we always assume $F$ to be a warping function as per Definition \ref{def:warpfun} and $w=(F^{-1})'$ the associated $v$-moderate weight. The resulting continuously indexed family of time-frequency atoms is given as follows.

\begin{definition}\label{def:warpedsystem}
  Let $F:D\rightarrow \RR$, $D\in\{\RR,\RR^+\}$ be a warping function and $\theta \in \bd L^2_{\sqrt{w}}(\RR)$. The \emph{continuous warped time-frequency system} with respect to $\theta$ and $F$ is defined by $\mathcal G(\theta,F):=\{g_{x,\xi}\}$, where 
  \begin{equation}
    g_{x,\xi} := T_{\xi} \widecheck{g_x},\quad g_x := \sqrt{F'(x)} (\bd T_{F(x)} \theta)\circ F \text{ for all } x\in D,\ \xi\in\RR.
  \end{equation}
  The \emph{phase space} associated with this family is $D \times \RR$.
\end{definition}

Clearly, $\mathcal G(\theta,F) \subset \mathcal F^{-1}\bd L^2(D)$, enabling the definition of a transform on $\mathcal F^{-1}(\LtD)$ by taking inner products with its elements.

\begin{definition}
The \emph{$F$-warped time-frequency transform} of $f\in \mathcal F^{-1}(\LtD)$ with respect to the warping function $F$ and the prototype $\theta\in \bd L^2_{\sqrt{w}}(\RR)$ is defined by
\begin{equation}\label{eq:warpedtransform}
  V_{\theta, F}f: D \times \RR \rightarrow \bbC, \;\; (x,\xi)\mapsto \inner{f}{g_{x,\xi}}.
\end{equation}
\end{definition}

Eq. \eqref{eq:moderatenormbound} and the above definition, immediately yield $V_{\theta, F}f \in \bd L^\infty(D\times \RR)$. However, using $F\in \mathcal C^1$ and translation-invariance of $\bd L^2_{\sqrt{w}}(\RR)$, we can deduce that also $V_{\theta,F}f\in \mathcal C(D\times \RR)$.

\begin{proposition}
  Let $F$ be a warping function and $\theta\in \bd L^2_{\sqrt{w}}(\RR)$. Then 
  \begin{equation}
    V_{\theta,F}f\in \mathcal C(D\times \RR) \cap \bd L^\infty(D\times \RR),\ \text{ for all } f\in \mathcal F^{-1}(\LtD). 
  \end{equation}
\end{proposition}
\begin{proof}
  By definition, $V_{\theta, F}f \in \bd L^\infty(D\times \RR)$. 
  We further compute the following estimate
  \begin{equation}
    \begin{split}
      |V_{\theta,F}f(x,\xi) &- V_{\theta,F}f(\tilde x, \tilde \xi)| = \abs{\inner{\hat f}{M_{-\xi} g_x - M_{-\tilde \xi} g_{\tilde x}}} \\
      & \leq \norm{\hat f}_{\LtD} \left( \norm{M_{-\xi} g_x - M_{-\tilde \xi}g_{x}}_{\LtD} + \norm{M_{-\tilde \xi}(g_x - g_{\tilde x})}_{\LtD} \right).
    \end{split}
    \label{}
  \end{equation}
  Since modulations are continuous on $\LtD$ it is sufficient to show that $\norm{g_{x}-g_{\tilde x}}_{\bd L^2}\rightarrow 0$, as $\tilde x$ tends to $x$.
  To see this we calculate
  \begin{equation}
    \begin{split}
      \norm{g_{x}-g_{\tilde x}}^2_{\bd L^2}&=\int_D \left| \sqrt{F'(x)}(T_{F(x)}\theta)(F(t))-\sqrt{F'(\tilde x)}(T_{F(\tilde x)}\theta)(F(t)) \right |^2 \; dt \\
      &= F'(x) \norm{T_{F(x)}\theta - \sqrt{F'(\tilde x)/F'(x)} T_{F(\tilde x)}\theta}^2_{\bd L^2_{\sqrt{w}}} \\
      &= F'(x) \norm{T_{F(x)}\theta -T_{F(\tilde x)}\theta +T_{F(\tilde x)}\theta - \sqrt{F'(\tilde x)/F'(x)} T_{F(\tilde x)}\theta}^2_{\bd L^2_{\sqrt{w}}}. 
  \end{split}
  \end{equation}
  Now a $2\epsilon$ argument finishes the proof since $\sqrt{F'(\tilde x)/F'(x)} \rightarrow 1$, $F(\tilde x) \rightarrow F(x)$ as $\tilde x \rightarrow x$ and translations are continuous on the weighted space $\bd L^2_{\sqrt{w}}$ due to moderateness of the weight function $w$.
\end{proof}

Indeed, $V_{\cdot, F}$ also possesses a norm-preserving property similar to the orthogonality relations (Moyal's formula~\cite{mo49,gr01}) for the short-time Fourier transform.

\begin{theorem}\label{thm:orthrel}
  Let $F$ be a warping function and $\theta_1, \theta_2 \in \bd L^2_{\sqrt{w}}$.
  Furthermore, assume that $\theta_1$ and $\theta_2$ fulfill the admissibility condition 
  \begin{equation}
    |\langle \theta_1,\theta_2 \rangle| <\infty.
    \label{eq:admissibilityC}
  \end{equation}
  Then the following holds for all $f_1,f_2 \in \mathcal F^{-1}(\LtD)$:
  \begin{equation}
    \int_D \int_{\RR} V_{\theta_1,F}f_1(x,\xi) \overline{V_{\theta_2,F}f_2(x,\xi)}\; d\xi dx = \inner{f_1}{f_2}\langle \theta_2,\theta_1 \rangle.
    \label{}
  \end{equation}
  In particular, if $\theta\in \bd L^2_{\sqrt{w}}$ is normalized in the (unweighted) $\bd L^2$ sense, then
  \begin{equation}
    \|V_{\theta,F} f\|_{\bd L^2(D\times \RR)} = \|f\|_{\bd L^2(\RR)}.
  \end{equation}
  \end{theorem}
\begin{proof}
  The elements of $\mathcal G(\theta_1,F)$ and $\mathcal G(\theta_2,F)$ will be denoted by $g^1_{x,\xi}$ and $g^2_{x,\xi}$, respectively.
  We use the fact that $V_{\theta_i,F}f(x,\xi)=\mathcal F^{-1}(\hat f \cdot \overline{g^i_{x}})(\xi)$ for $i=1,2$ to calculate
  \begin{equation}
    \begin{split}
      \int_D \int_\RR & V_{\theta_1,F}f_1(x,\xi)\overline{V_{\theta_2,F}f_2(x,\xi)} \; d\xi dx \\
      &=\int_D \int_\RR \mathcal F^{-1}(\hat f_1 \cdot \overline{g^1_{x}})(\xi) \overline{\mathcal F^{-1}(\hat f_2 \cdot \overline{g^2_{x}})(\xi)} \; d\xi dx\\
      &= \int_D \hat f_1(t) \overline{\hat f_2(t)} \int_D \overline{g^1_{x}(t)} g^2_{x}(t) \; dx dt
    \end{split}
    \label{}
  \end{equation}
  Using the substitution $u=F(t)-F(x)$ we can simplify the inner integral
  \begin{equation}
    \begin{split}
      \int_D \overline{g^1_{x}(t)} g^2_{x}(t) \; dx &= \int_D F'(x) \overline{\theta_1(F(t)-F(x))} \theta_2(F(t)-F(x)) \; dx \\
      &=\int_\RR \overline{\theta_1(u)} \theta_2(u) \; du = \langle \theta_2,\theta_1 \rangle.
    \end{split}
    \label{}
  \end{equation}
  The desired results follow using Parseval's formula (and setting $f_1 = f_2 = f$ and $\theta_1 = \theta_2 = \theta$).
\end{proof}

The orthogonality relations are tremendously important, because they immediately yield an inversion formula for $V_{\theta, F}$, similar to the inversion formula for wavelets and the STFT. They even imply that $\{g_{x,\xi}\}_{x\in D,\xi\in\RR}$ forms a continuous tight frame with frame bound $\|\theta\|_{\bd L^2}^2$. Note that the admissibility condition Eq. \eqref{eq:admissibilityC} is always satisfied if $D=\RR$. In that case $w=(F^{-1})'$ is bounded below, implying $\bd L^2_{\sqrt{w}} \subseteq \bd L^2$. On the other hand, if $D=\RR^+$, $w$ can never be bounded below and the admissibility condition is a real restriction. Moreover, for $F=\log$, $\theta_1,\theta_2\in\LtR$ is equivalent to $g^1_{F^{-1}(0),0},g^2_{F^{-1}(0),0}$ being admissible wavelets, i.e. $g^1_{F^{-1}(0),0},g^2_{F^{-1}(0),0}$ satisfy the classical wavelet admissibility condition.

\begin{corollary}
  Given a warping function $F$ and some nonzero $\theta\in \bd L^2_{\sqrt{w}}\cap \bd L^2$.
  Then any $f \in \mathcal F^{-1}(\LtD)$ can be reconstructed from $V_{\theta, F}f$ by
  \begin{equation}
    f = \frac{1}{\|\theta\|_{\bd L^2}} \int_D \int_\RR V_{\theta,F}f(x,\xi) g_{x,\xi} \; d\xi dx.
    \label{}
  \end{equation}
  The equation holds in the weak sense.
\end{corollary}
\begin{proof}
  The assertion follows easily from the orthogonality relations by setting $\theta = \theta_1=\theta_2$ since
  for any given $f_2 \in \mathcal F^{-1}(\LtD)$ we have the relation
  \begin{equation}
    \inner{f}{f_2} = \frac{1}{\|\theta\|_{\bd L^2}} \int_D \int_\RR \inner{f}{ g_{x,\xi}}\inner{g_{x,\xi}}{f_2} \; d\xi dx.
    \label{}
  \end{equation}
\end{proof}

To conclude this section we give some examples of warping functions that are of particular interest, as they encompass important frequency scales. For a proof that the presented examples indeed define warping functions in the sense of Definition \ref{def:warpfun}, please see~\cite[Proposition 2]{howi14}.

\begin{example}[Wavelets]\label{ex:wavelet1}
  Choosing $F=\log$, with $D=\RR^+$ leads to a system of the form
  \begin{equation}
    g_{x}(t) = x^{-1/2}\theta(\log(t)-\log(x)) = x^{-1/2}\theta(\log(t/x)) = x^{-1/2}g_{F^{-1}(0)}(t/x).
    \label{}
  \end{equation}
  This warping function therefore leads to $g_x$ being a dilated version of $g_{1}$. 
  Note the interaction of the Fourier transform and dilation to see that $\mathcal G(\theta, \log)$ is indeed a continuous wavelet system, with the minor modification that our scales are reciprocal to the usual definition of wavelets.
\end{example}

\begin{example}\label{ex:Rplusalpha1}
  The family of warping functions $F_l(t) = c\left((t/d)^l - (t/d)^{-l}\right)$, for some $c,d >0$ and $l \in ]0,1]$, is an alternative to the logarithmic warping for the domain $D=\RR^+$. The logarithmic warping in the previous example can be interpreted as the limit of this family for $l\rightarrow 0$ in the sense that for any fixed $t\in \RR^+$,
  \begin{equation}
    F_l'(t) = \frac{lc}{d}\left((t/d)^{-1+l} + (t/d)^{-1-l}\right) \overset{l\rightarrow 0}{\rightarrow} \frac{2lc}{t} = \frac{2lc}{d}\log'(t/d).
  \end{equation}

  This type of warping provides a frequency scale that approaches the limits $0,\infty$ of the frequency range $D$ in a slower fashion than the wavelet warping. In other words, $g_x$ is less deformed for $x>F_l^{-1}(0)$, but more deformed for $x<F_l^{-1}(0)$ than in the case $F=\log$. On the other hand, the property that $g_x$ can be expressed as dilation of $g_{F^{-1}_l(0)}$, or any other unitary operator applied to $g_{F^{-1}_l(0)}$, is lost.
\end{example}

\begin{example}[ERBlets]\label{ex:ERB1}
  In psychoacoustics the investigation of filter banks adapted to the spectral resolution of the human ear has been subject to a wealth of research, see \cite{momo03} for an overview.
  We mention here the Equivalent Rectangular Bandwidth scale (ERB-scale) described in \cite{glmo90}, which introduces a set of bandpass filters following the human perception.
  In \cite{bahoneso13} the authors construct a filter bank that is designed to be adapted to the ERB-scale. The ERB warping function, given by
  \begin{equation}
      F_{\text {ERB}}(t) = \sgn{(t)} \; c_1 \log\left(1+\frac{\abs{t}}{c_2}\right), \\
    \label{}
  \end{equation}
  can also be used to construct a continuous time-frequency representation on an auditory scale. The ERB-scale is obtained for $c_1=9.265$ and $c_2=228.8$.
  Being adapted to the human perception of sound, this representation has potential applications in sound signal processing.
\end{example}

\begin{example}\label{ex:alpha1}
  The warping function $F_l(t) = \sgn(t)  \,\left((\abs{t}+1)^l-1\right)$ for some $l \in ]0,1]$ leads to a transform that is structurally very similar to the \emph{$\alpha$-transform}
  . Much in the same way, this family of warping functions can be seen as an interpolation between the identity ($l=1$), which leads to the STFT, and an ERB-like frequency scale for $l\rightarrow 0$.
  This can be seen by differentiating $F$ and observing that for $l$ approaching $1$ this derivative approaches (up to a factor) the derivative of the ERB warping function for $c_1=c_2=1$. The connection between this type of warping and the $\alpha$-transform is detailed below.
  
  The $\alpha$-transform
  , provides a family of time-frequency transforms with varying time-frequency resolution. its time-frequency atoms are constructed from a single prototype by a combination of translation, modulation and dilation:
  \begin{equation}
    g_{x,\xi}(t) = \eta_\alpha(x)^{1/2} e^{2\pi i x t}g\left(\eta_\alpha(x) (t-\xi) \right),
  \end{equation}
  for $\eta_\alpha(x) = (1+|x|)^\alpha$, with $\alpha\in [0,1]$. If $\mathcal{F}g$ is a symmetric bump function centered at frequency $0$, with a bandwidth\footnote{The exact definition of bandwidth, e.g. frequency support or $-3$db bandwidth, is not important for this example.} of $1$, then $\mathcal{F}g_{x,0}$ is a symmetric bump function centered at frequency $x$, with a bandwidth of $\eta_\alpha(x)$. Up to a phase factor, $\mathcal{F}g_{x,\xi} = \bd M_{-\xi}\mathcal{F}g_{x,0}$. Varying $\alpha$, one can \emph{interpolate} between the STFT ($\alpha = 0$, constant time-frequency resolution) and a wavelet-like (or more precisely ERB-like) transform with the dilation depending linearly on the center frequency ($\alpha=1$).
  
  Through our construction, we can obtain a transform with similar properties by using the warping functions $F(t) = l^{-1}\sgn(t)\left((1+|t|)^{l}-1\right)$, for $l\in ]0,1]$, and $F(t) =\sgn(t)\log(1+|t|)$, introduced here and in Example \ref{ex:ERB1}. Take $\theta$ a symmetric bump function centered at frequency $0$, with a bandwidth of $1$. Then $\mathcal{F}g_{x,0} = \sqrt{F'(x)}\theta(F(t)-F(x))$ is still a bump function with peak frequency $x$, but only symmetric if $l=1$ or $x=0$. Moreover, the bandwidth of $\mathcal{F}g_{x,0}$ equals 
  \begin{equation}
    F^{-1}(F(x)+1/2) - F^{-1}(F(x)-1/2) = \int_{-1/2}^{1/2} (F^{-1})'(F(x)+s)\; ds \approx 1/F'(x).
  \end{equation}
  Note that $F'(x) = (1+|x|)^{l-1} = 1/\eta_{1-l}(x)$, for $l\in ]0,1]$, and for $F(t) =\sgn(t)\log(1+|t|)$ we obtain $F'(x) = (1+|x|)^{-1} = 1/\eta_1(x)$. Finally, $\mathcal{F}g_{x,\xi} = \bd M_{-\xi}\mathcal{F}g_{x,0}$. All in all, it can be expected that the obtained warped transforms provide a time-frequency representation very similar to the $\alpha$-transform with the corresponding choice of $\alpha$.
\end{example}

\section{Coorbit spaces for warped time-frequency systems}\label{sec:warpedcoorbits}

   In the previous section we have developed time-frequency representations for functions $f\in \mathcal F^{-1}\bd L^2\left( D \right)$.
   Due to the inner product structure of the coefficient computation it seems natural to attempt the representation of distributions $f$ by restricting the pool of possible functions $\theta$, so that the resulting warped filter bank consists entirely of suitable test functions. In the setting of classical Gabor and wavelet transforms, the appropriate setting is Feichtinger and Gr\"ochenig's coorbit space theory~\cite{fegr89,fegr89-1}.
  
  In addition to a Banach space of test functions and the appropriate distribution space, coorbit theory provides a complete family of (nested) Banach spaces, the elements of which are characterized by their decay properties in the associated time-frequency representation. However, most attempts to generalize coorbit space theory still require the examined TF representation to be based on an underlying group structure, similar to the STFT being based on the (reduced) Heisenberg group in the classical theory. 
  
  Since our warped TF transform $V_{\theta,F}$ does not possess such a structure, the appropriate framework for the construction of the associated coorbit spaces is the generalized coorbit theory by Fornasier and Rauhut~\cite{fora05}. In other words, we aim to translate the results presented in Section \ref{ssec:coorbit} to the setting of warped time-frequency systems. The first step towards this is finding sufficient conditions for a prototype function $\theta$, such that $\mathcal G(\theta,F)$ satisfies $K_{\theta,F}:= K_{\mathcal G(\theta,F)}\in \AAm$, for suitable weights $m$. A large part of this section is devoted to proving the following main result.
   
   \begin{theorem}\label{thm:mainres1}
   Let $F~:~D\mapsto \RR$ be a warping function with $w=(F^{-1})'\in\mathcal{C}^1(\RR)$, such that for all $x,y\in\RR$:
   \begin{equation}\label{eq:warpfunconds}
      \frac{w(x+y)}{w(x)w(y)} \leq C<\infty \text{ and } \left|\frac{w'}{w}\right|(x)\leq D_1<\infty.
   \end{equation}
   Furthermore, let $m_1:D\rightarrow \RR$ such that $m_1\circ F^{-1}$ is $v_1$-moderate, for a symmetric, submultiplicative weight function $v_1$ and define $m(x,y,\xi,\omega) = \max\left\{\frac{m_1(x)}{m_1(y)},\frac{m_1(y)}{m_1(x)}\right\}$. Then
   \begin{equation}
     K_{\theta,F}\in \AAm,\quad \text{ for all } \theta\in\mathcal C^\infty_c.
   \end{equation}
   If furthermore $w,v_1\in \mathcal O\left((1+|\cdot|)^p\right)$ for some $p\in\RR^+$, then
   \begin{equation}
     K_{\theta,F}\in \AAm,\quad \text{ for all } \theta\in \mathcal S.
   \end{equation} 
   \end{theorem}
   
   In fact, we prove a stronger result providing weaker, but more technical conditions on $F$, $m$ and $\theta$. As indicated by \eqref{eq:warpfunconds}, we will not be able to construct coorbit spaces for arbitrary warping functions. Indeed, we require $w=(F^{-1})'$ to be \emph{almost submultiplicative}, i.e. 
   \begin{equation}\label{eq:quasisubmult}
     w(x+y) \leq Cw(x)w(y), \text{ for all } x,y\in\RR
   \end{equation}
   and a suitable $C>0$. For the remainder of this manuscript, we will assume Eq. \eqref{eq:quasisubmult} to hold, without loss of generality we also assume $C\geq 1$.
   
   \begin{remark}
     Since we require the warping functions to satisfy Eq. \eqref{eq:quasisubmult}, the following results do not hold for the warping functions $F_l(x) = x^l-x^{-l}$, $l\in [0,1[$, see Example \ref{ex:Rplusalpha1}. It might still be possible to construct coorbit spaces for that type of warping function, albeit not with the methods presented here. 
   \end{remark}
   
   \begin{remark}
     The warping functions $F(t) = \sgn(t)\log(1+|t|)$ and $F(t) = \sgn(t)\left((1+|t|)^l-1\right)$, $l\in ]0,1]$, see Examples \ref{ex:ERB1} and \ref{ex:alpha1}, are $\mathcal C^\infty$ only on $\RR\setminus\{0\}$. For smoothness of $F^{-1}$ at $t=0$ select a function $G\in\mathcal{C}^\infty_c$ with $\supp(G) = ]-\epsilon,\epsilon[$ and construct a smooth transition $\tilde{F}(t)=(G(0)-G(t))F(t)+G(t)F(\epsilon)t/\epsilon$ between $F$ and the identity. Then $\tilde F^{-1}(s)\in\mathcal C^{\infty}(\RR)$ and $\tilde F = F$ on $\RR\setminus ]-\epsilon,\epsilon[$.
   \end{remark}

   We begin by noting that the norm condition $\|K_{\theta,F}\|_{\AAm} < \infty$ reduces to 
   \begin{equation}
     \esssup_{x,\xi\in\RR} I_{\theta,F,m}(x,\xi) < \infty,
   \end{equation}
   where 
   \begin{equation}\label{eq:amclass}
     I_{\theta,F,m}(x,\xi) := \int_\RR \int_\RR \left| \int_\RR C_x(z) \tilde{m}(x,z,\xi,\eta)\frac{w(s+x)}{w(x)} \theta(s)\overline{\bd T_{z}\theta}(s)e^{-2\pi i \eta\frac{F^{-1}(s+x)}{w(x)}} ~ds\right|~d\eta~dz,
   \end{equation}
   and $C_x(z) = \sqrt{\frac{w(z+x)}{w(x)}}\leq \sqrt{Cw(z)}$ and $\tilde{m}(x,z,\xi,\eta) = m\left(F^{-1}(x),F^{-1}(z+x),\xi,\xi-\frac{\eta}{w(x)}\right)$.
   
   The expression describing $I_{\theta,F,m}(x,\xi)$ is obtained by (i) inserting the definition of $K_{\theta,F}$ and $g_{x,\xi}$ while substituting $x \rightarrow F^{-1} (x)$, and (ii) performing a triple substitution:
   \begin{equation}
     \begin{split}
       \lefteqn{\int_D \int_\RR m(F^{-1}(x),y,\xi,\omega) |\langle g_{F^{-1}(x),\xi},g_{y,\omega}\rangle|\; d\omega\; dy = \int_D \int_\RR m(F^{-1}(x),y,\xi,\omega) |\langle \widehat{g_{F^{-1}(x),\xi}},\widehat{g_{y,\omega}}\rangle|\; d\omega\; dy}\\
       & = \int_D \int_\RR m(F^{-1}(x),y,\xi,\omega) \left|\int_D \sqrt{F'(F^{-1}(x))F'(y)}\theta(F(t)-x)\overline{\theta(F(t)-F(y))}e^{-2\pi i (\xi-\omega)t}~dt\right|~d\omega~dy\\
       &\overset{(1)}{=} \int_\RR C_x(z) \int_\RR m(F^{-1}(x),F^{-1}(z+x),\xi,\omega) \left| \int_\RR w(s+x) \theta(s)\overline{\bd T_{z}\theta}(s)e^{-2\pi i (\xi-\omega) F^{-1}(s+x)} ~ds\right|~d\omega~dz\\
       &\overset{(2)}{=} I_{\theta,F,m}(x,\xi),
  \end{split}
\end{equation}
   where $(1)$ follows by the substitutions $s = F(t)-x$ and $z = F(y)-x$, whereas $(2)$ follows by the substitution $\eta = w(x)(\xi-\omega)$. In order to obtain an estimate $I_{\theta,F,m} < \tilde{C} < \infty$, we derive an estimate for the innermost integral, such that the outer integrals converge. The most important tool for that purpose is the so-called method of stationary phase~\cite{st93}, a particular case of partial integration most widely known for being the classical method of proving $f\in\mathcal C^p_0 \Rightarrow \hat{f}\in \mathcal O((1+|\cdot|)^{-p})$, for $p\in\NN$. In our case, it can be formulated as 
   \begin{equation} \label{sec:statphas1}
     \int_\RR f(s) e^{-2\pi i \eta \frac{F^{-1}(s+x)}{w(x)}}~ds = \int_\RR \mathbf{D}_{w,x,\eta}(f)(s) e^{-2\pi i \eta \frac{F^{-1}(s+x)}{w(x)}} ~ds,
   \end{equation}
   for all $f\in\mathcal C^1$, with $wf'\in\mathcal C_0$. 
   Here, we use the operators $\mathbf{D}_{w,x,\eta}~:~\mathcal{C}^1(\RR)\mapsto \mathcal{C}(\RR)$ defined 
   by
   \begin{equation}
    \mathbf{D}_{w,x,\eta}(f)(s) = \frac{w(x)}{2\pi i \eta} \left(\frac{f}{\bd T_{-x}w}\right)'(s).
   \end{equation}
       
We need to prove some auxiliary results for those operators:

   \begin{lemma}\label{lem:Dnoperator}
  Let $w\in\mathcal{C}^{n}(\RR)$, for some $n\in\NN_0$. 
  Then, $\mathbf{D}^n_{w,x,\eta}~:~\mathcal{C}^n(\RR)\mapsto \mathcal{C}(\RR)$ is given by
  \begin{equation}\label{eq:Dnoperator}
    \mathbf{D}^n_{w,x,\eta}(f)(s) = \frac{w(x)^n}{(2\pi i \eta)^n} \bd T_{-x}w(s)^{-2n} \sum_{k=0}^n f^{(k)}(s)\bd P_{n,k}(s+x), \text{ for all } f\in\mathcal{C}^n(\RR),\ s\in\RR,
  \end{equation}
  where
  \begin{equation}\label{eq:Pnkdefine}
    \bd P_{n,k}(s) = \sum_{\sigma\in\Sigma_{n,k}} C_\sigma \prod_{j=1}^n w^{(\sigma_j)}(s),
  \end{equation}
  with $\Sigma_{n,k} \subseteq \left\{ \sigma = (\sigma_1,\ldots,\sigma_n)~:~ \sigma_j\in\{0,\ldots,n-k\}\right\}$ and $C_\sigma \in \RR$ for all $\sigma\in\Sigma_{n,k}$.
\end{lemma}

\begin{proof}
  The assertion is proven by induction. By definition, 
  \begin{equation}
    \mathbf{D}_{w,x,\eta}(f)(s) = \frac{w(x)}{2\pi i \eta} \bd T_{-x}w(s)^{-2} \left(-f(s)w'(s+x)+f'(s)\bd T_{-x}w(s)\right).
  \end{equation}
  Therefore, the assertion holds for $n=1$. For the induction step, note that
  \begin{equation}
      \mathbf{D}^{n}_{w,x,\eta}(f)(s) = \mathbf{D}_{w,x,\eta}(\mathbf{D}^{n-1}_{w,x,\eta}(f))(s) = \frac{w(x)^{n}}{(2\pi i \eta)^{n}}\underbrace{\left(\frac{\sum_{k=0}^{n-1} f^{(k)}\bd T_{-x}\bd P_{n-1,k}}{\bd T_{-x}w^{2n-1}}\right)'(s)}_{=:G_{w,x}(s)}.
  \end{equation}
  By the quotient rule, 
  \begin{equation}
    \begin{split}
    G_{w,x}(s) & = \bd T_{-x}w(s)^{-2n}\sum_{k=0}^{n-1} \bigg(f^{(k+1)}(s)\bd P_{n-1,k}(s+x)\bd T_{-x}w(s)\\
    & + f^{(k)}(s)\bd P_{n-1,k}'(s+x)\bd T_{-x}w(s) - (2n-1)f^{(k)}(s)\bd P_{n-1,k}(s+x)w'(s+x)\bigg).
    \end{split}
  \end{equation}
  Using the definition of $\bd P_{n-1,k}$ it is easy to see that $\bd P_{n-1,k}(s+x)\bd T_{-x}w(s)$ and $\bd P_{n-1,k}(s+x)w'(s+x)$ are sums of $n$-term products of $w$ and its derivatives of order no higher than $n-k-1$ and $n-k$, respectively. Furthermore, the highest order derivative of $f$ appearing in $G_{w,x}$ is $f^{(n)}$. It remains to show that $\bd P_{n-1,k}'$ is a sum of $(n-1)$-term products of $w$ and its derivatives of order no higher than $n-k$. For any term of $\bd P_{n-1,k}$, i.e. any $\sigma\in\Sigma_{n-1,k}$, we obtain 
  \begin{equation}
    \left(\prod_{j=1}^{n-1} w^{(\sigma_j)}\right)'(s) = \sum_{j=1}^{n-1} w^{(\sigma_j+1)}(s)\prod_{l\in\{1,\ldots,n-1\}\setminus\{j\}}w^{(\sigma_l)}(s).
  \end{equation}
  Therefore, the idividual terms of $G_{w,x}$ satisfy the conditions imposed on the terms of $\bd P_{n,k}$ and reordering them by the appearing derivative of $f$ completes the proof.
\end{proof}

The following corollary shows that $\mathbf{D}^{n}_{w,x,\eta}(f)(s)$ is uniformly bounded as a function in $x\in\RR$, under suitable assumptions on the function $w$.

\begin{corollary}\label{cor:Dnestimate}
  Let $w$ satisfy Eq. \eqref{eq:quasisubmult} and $w\in\mathcal{C}^n(\RR)$. If there are constants $D_k > 0$, $k=0,\ldots,n$, such that $|w^{(k)}/w|(s)\leq D_k$, then there is a finite constant
  $C_n>0$, such that 
  \begin{equation}\label{eq:cor1assert}
    |\mathbf{D}^{n}_{w,x,\eta}(f)(s)|\leq C_n \left(\frac{w(-s)}{2\pi |\eta|}\right)^n \sum_{k=0}^n |f^{(k)}(s)|, \text{ for all } s,x,\eta\in\RR \text{ and all } f\in\mathcal C^n(\RR).
  \end{equation}
\end{corollary}
\begin{proof}
  First, note that Eq. \eqref{eq:quasisubmult} ensures $w(x)/w(s+x) \leq Cw(-s)$. Invoke Lemma \ref{lem:Dnoperator} to see that 
  \begin{equation}\label{eq:Pquotmax}
    C_n:= \max\limits_{k=0}^n \sup\limits_{s\in\RR} \left|\bd P_{n,k}(s)/w(s)^n\right|
  \end{equation}
  is a viable choice for Eq. \eqref{eq:cor1assert} to hold, provided it is finite. Use Lemma \ref{lem:Dnoperator} again to obtain
  \begin{equation}
    \left|\frac{\bd P_{n,k}(s)}{w(s)^n}\right| = \sum_{\sigma\in\Sigma_{n,k}} |C_\sigma| \prod_{j=1}^n \left|\frac{w^{(\sigma_j)}(s)}{w(s)}\right| \leq \sum_{\sigma\in\Sigma_{n,k}} |C_\sigma| \prod_{j=1}^n D_{\sigma_j+1} < \infty.
  \end{equation}
  Since the sets $\{\bd P_{n,k}\}$ and $\Sigma_{n,k}$ are finite, the expression in Eq. \eqref{eq:Pquotmax} is finite and there is a finite $C_n>0$.
\end{proof}

\begin{lemma}\label{lem:cnexist}
  Let $F$ be a warping function such that $w=(F^{-1})'\in \mathcal{C}^n(\RR)$ satisfies Eq. \eqref{eq:quasisubmult} and $|w^{(k)}/w|\leq D_k$ for some constants $D_k > 0$, $k=0,\ldots,n$. If $f\in \mathcal C^{n+1}(\RR)\cap \bd L^1_{w}(\RR)$, with 
  \begin{equation}\label{eq:cjlims}
    f,\ w(-\cdot)^j f^{(k+1)}\in\mathcal C_0(\RR), \text{ for all } 0\leq k\leq j,\ 0\leq j\leq n
  \end{equation}
  and, with $C_n$ as in Corollary \ref{cor:Dnestimate},
  \begin{equation}\label{eq:cjs}
     C_n\int_\RR w(-s)^n|f^{(k+1)}(s)|~ds \leq c_n < \infty , \text{ for all } 0\leq k\leq n.
  \end{equation}
  Then 
  \begin{equation}\label{eq:cnestimate}
    \left|\int_\RR \frac{w(s+x)}{w(x)}f(s) e^{-2\pi i \eta \frac{F^{-1}(s+x)}{w(x)}}~ds\right| \leq \frac{(n+1)c_n}{(2\pi|\eta|)^{n+1}}.
  \end{equation}
  Furthermore, the LHS of Eq. \eqref{eq:cnestimate} is bounded by $C\|f\|_{\bd L^1_{w}(\RR)}$.
\end{lemma}
\begin{proof}
  To prove the second assertion, simply note that 
  \begin{equation}
    \left|\frac{w(s+x)}{w(x)}f(s)\right| \leq C|w(s)f(s)|
  \end{equation}
  by Eq. \eqref{eq:quasisubmult}.
  Now note that 
  \begin{equation}
    \mathbf{D}_{w,x,\eta}\left(\frac{w(\cdot+x)}{w(x)}f\right)(s) = (2\pi i \eta)^{-1}f'(s).
  \end{equation}
  Combine Eq. \eqref{eq:cjlims} with Corollary \ref{cor:Dnestimate} and the stationary phase method to find that
  \begin{equation}
    \begin{split}
      \lefteqn{\left|\int_\RR \frac{w(s+x)}{w(x)}f(s) e^{-2\pi i \eta \frac{F^{-1}(s+x)}{w(x)}}~ds\right|}\\
      & = \left|(2\pi i\eta)^{-1}\int_\RR \mathbf{D}^{n}_{w,x,\eta}(f')(s) e^{-2\pi i \eta \frac{F^{-1}(s+x)}{w(x)}}~ds\right|\\
      & \overset{\text{Cor.\ref{cor:Dnestimate}}}{\leq} C_n(2\pi |\eta|)^{-(n+1)} \int_\RR w(-s)^n \sum_{k=1}^{n+1} |f^{(k)}(s)|~ds \leq \frac{(n+1)c_n}{(2\pi |\eta|)^{n+1}},
    \end{split}
  \end{equation}
  where we used Eq. \eqref{eq:cjs} to obtain the final estimate. This completes the proof.
\end{proof}

In our specific case, the function $f$ has a special form, namely $f(s) = \theta(s)\overline{\bd T_z\theta}(s)$. We now determine conditions on $\theta$ such that the estimates obtained through Lemma \ref{lem:cnexist} are integrable. This is the final step for establishing convergence of the triple integral Eq. \eqref{eq:amclass}.

\begin{lemma}\label{lem:cznconvergence}
  Let $F$ be a warping function such that $w=(F^{-1})'\in \mathcal{C}(\RR)$ satisfies Eq. \eqref{eq:quasisubmult}. Let furthermore $v$ be a symmetric, submultiplicative weight function. If $\theta\in\mathcal C^{n+1}(\RR)$ such that for all $0\leq k\leq n+1$
  \begin{equation}  
    \theta^{(k)}\in \bd L^2_{w_1}(\RR)\cap \bd L^2_{w_2}(\RR), 
  \end{equation}
  with
  \begin{equation}
   w_1(s):= v(s)w(-s)^n\sqrt{w(s)}(1+|s|)^{1+\epsilon} \text{ and } w_2(s):= v(s)\sqrt{w(-s)}(1+|s|)^{1+\epsilon},
  \end{equation}
  then 
  \begin{equation}\label{eq:dzintegral}
    \int_\RR \sqrt{w(z)}v(z) c_n(z)~dz < \infty.
  \end{equation}
  Here, 
  \begin{equation}
    c_n(z) := \max_{k=1}^{n+1}
    \left(C_n \int_\RR w(-s)^n |\left(\theta\overline{\bd T_z\theta}\right)^{(k)}(s)|~ds\right).
  \end{equation} 
\end{lemma}
\begin{proof}
  For ease of notation, we will denote in the following chain of inequalities by $\tilde{C}$ a finite product of nonnegative, finite constants. Therefore, $\tilde{C}$ might have a different value  after each derivation, but is always nonnegative and finite. There is $1\leq k\leq n+1$, such that the LHS of Eq. \eqref{eq:dzintegral} equals
  \begin{equation}
    \begin{split}
      \lefteqn{\tilde{C} \int_\RR \sqrt{w(z)}v(z) \left|\int_\RR w(-s)^n \left(\theta\overline{\bd T_z\theta}\right)^{(k)}(s)~ds\right|~dz}\\
      & \leq \tilde{C} \int_\RR\int_\RR \sqrt{w(z)}v(z)w(-s)^n  \sum_{j=0}^k |\theta^{(j)}(s)||\bd T_z\theta^{(k-j)}(s)|~ds~dz\\
      & \leq \tilde{C} \int_\RR \sum_{j=0}^k \int_\RR |w(s)^{\frac{1}{2}}v(s)w(-s)^n \theta^{(j)}(s)||w(z-s)^{\frac{1}{2}}v(z-s)\theta^{(k-j)}(s-z)|~ds~dz\\
      & \leq \tilde{C} \int_\RR (1+|z|)^{-1-\epsilon} \sum_{j=0}^k \int_\RR  |w_1(s)\theta^{(j)}(s)||w_2(s-z)\theta^{(k-j)}(s-z)|~ds~dz\\
      & \leq \tilde{C} \max_{j=0}^k \left( \|\theta^{(j)}\|_{\bd L^2_{w_1}(\RR)}\|\theta^{(k-j)}\|_{\bd L^2_{w_2}(\RR)} \right) \int_\RR (1+|z|)^{-1-\epsilon}~dz < \infty.
    \end{split}
  \end{equation}
  In this derivation, we used Eq. \eqref{eq:quasisubmult} repeatedly, as well as submultiplicativity and symmetry of $(1+|s|)^{-1-\epsilon}$ and symmetry of $v$. Furthermore, we used the product rule for differentiation and that the appearing sum is finite. In the final step, we applied Cauchy-Schwarz' inequality.
\end{proof}

We are now ready to prove the main result simply by collecting the conditions from the interim results above. The proof itself is only little more than sequentially applying those interim results to the function $I_{\theta,F,m}$ given by Eq. \eqref{eq:amclass}.


\begin{theorem}\label{thm:thetainAm}
  Let $F$ be a warping function such that $w=(F^{-1})'\in \mathcal{C}^{p+1}(\RR)$ satisfies Eq. \eqref{eq:quasisubmult} and $|w^{(k)}/w|\leq D_k$ for some constants $D_k > 0$, $k=0,\ldots,p+1$, where $p\in\NN$ if $D=\RR$ and $p=0$ if $D=\RR^+$.
  Furthermore, let 
  \begin{equation} 
    m(x,y,\xi,\omega) := \max \left\{\frac{m_1(x)m_2(\xi)}{m_1(y)m_2(\omega)},\frac{m_1(y)m_2(\omega)}{m_1(x)m_2(\xi)}\right\},
  \end{equation}
  with weight functions $m_1,m_2$ that satisfy
  \begin{itemize}
   \item[(i)] $m_1\circ F^{-1}$ is $v_1$-moderate, for a symmetric, submultiplicative weight function $v_1$ and
   \item[(ii)] $m_2$ is $v_2$-moderate, for a symmetric, submultiplicative weight function $v_2
   \in \mathcal O\left((1+|\cdot|)^p \right)$.
  \end{itemize}
  Assume that $\theta\in\mathcal C^{p+2}(\RR)\cap \bd L^2_{\sqrt{w}}(\RR)$ satisfies
  \begin{itemize}
      \item[(a)] $\theta, w(-\cdot)^j\theta^{(k+1)}\in \mathcal C_0(\RR)$, for all $0\leq k\leq j,\ 0\leq j\leq p+1$,
   \item[(b)] $\theta\in \bd L^2_{w_1}(\RR)\cap \bd L^2_{w_2}(\RR)$, with $w_1(s) = v_1(s)(1+|s|)^{1+\epsilon}w(-s)^{1/2}$, $w_2(s) = w_1(-s)w(s)$ and
   \item[(c)] $\theta^{(k)}\in \bd L^2_{w_1}(\RR)\cap \bd L^2_{w_3}(\RR)$ for all $0\leq k \leq p+2$, with $w_3(s) = w_1(-s)w(-s)^{p+1}$,
  \end{itemize}
  for some $\epsilon > 0$. 
  Then
  \begin{equation}\label{eq:amclass0}
    \esssup_{x,\xi\in\RR} I_{\theta,F,m} < \infty \text{ and therefore } K_{\theta,F}\in\AAm.
  \end{equation}
\end{theorem}
\begin{proof}
     Recall the Definition of $I_{\theta,F,m}$ in Eq. \eqref{eq:amclass}
     \begin{equation*}
     I_{\theta,F,m}(x,\xi) = \int_\RR \int_\RR \left| \int_\RR C_x(z) \tilde{m}(x,z,\xi,\eta)\frac{w(s+x)}{w(x)} \theta(s)\overline{\bd T_{z}\theta}(s)e^{-2\pi i \eta\frac{F^{-1}(s+x)}{w(x)}} ~ds\right|~d\omega~dz,
   \end{equation*}
   where $C_x(z) = \sqrt{\frac{w(z+x)}{w(x)}}$ and $\tilde{m}(x,z,\xi,\eta) = m\left(F^{-1}(x),F^{-1}(z+x),\xi,\xi - \frac{\eta}{w(x)}\right)$.
   
   We already know that $C_x(z) = \sqrt{\frac{w(z+x)}{w(x)}}\leq \sqrt{Cw(z)}$ by the assumptions on $F$. We now estimate the time-frequency weight $\tilde m$. To that end, observe
   \begin{equation}
     \begin{split}
	\lefteqn{\tilde m(x,z,\xi,\eta)}\\
	& \leq \max \left\{\frac{m_1(F^{-1}(x))}{m_1(F^{-1}(z+x))},\frac{m_1(F^{-1}(z+x))}{m_1(F^{-1}(x))}\right\}\max \left\{\frac{m_2(\xi)}{m_2(\xi-\eta/w(x))},\frac{m_2(\xi-\eta/w(x))}{m_2(\xi)}\right\}\\
	& \leq \tilde{C}_1v_1(z)\tilde{C}_2v_2\left(\eta/w(x)\right)\\
	& \leq \tilde{C}_1\tilde{C}_2v_1(z)V_2(\eta), \text{ where } V_2(\eta):=\begin{cases}
	                                          \max_{|u|\leq \eta} v_2\left(u/w(0)\right) & \text{ if } D=\RR,\\
	                                          1 			  & \text{ if } D=\RR^+,
						  \end{cases}
     \end{split}
   \end{equation}
   Here we used Conditions (i) and (ii) on $m_1,m_2$. For the final inequality, we used that $w$ is nondecreasing on $\RR^+$ and, if $D=\RR$, symmetric. Note that for $D = \RR^+$ we have $p = 0$.
   
   For sufficiently large $\tilde{C}>0$, 
   \begin{align}
     \lefteqn{I_{\theta,F,m}(x,\xi) }\nonumber\\
     & \leq \tilde{C}\int_\RR\int_\RR \left|\int_\RR \sqrt{w(z)}v_1(z)V_2(\eta) \frac{w(s+x)}{w(x)}\theta(s)\overline{\bd T_z\theta}(s) e^{-2\pi i \eta \frac{F^{-1}(s+x)}{w(x)}}~ds\right|~d\eta ~dz \label{eq:0thorder}\\
     & = \tilde{C}\int_\RR\int_\RR V_2(\eta) \left|\int_\RR \sqrt{w(z)}v_1(z) \mathbf{D}^{p+2}_{w,x,\eta}\left(\frac{w(\cdot+x)}{w(x)}\theta\overline{\bd T_z\theta}\right)(s) e^{-2\pi i \eta \frac{F^{-1}(s+x)}{w(x)}}~ds\right|~d\eta ~dz,\label{eq:pthorder}
   \end{align}
   where we used the method of stationary phase, together with condition (a) on $\theta$. 
   
    To obtain the final estimate, we distinguish between the cases $|\eta| < 1$ and $|\eta|\geq 1$. In the first case, we use the expression in Eq. \eqref{eq:0thorder} and obtain
   \begin{equation}\label{eq:smalleta}
     \begin{split}
       \lefteqn{\int_\RR\int_{-1}^1 V_2(\eta) \left|\int_\RR \sqrt{w(z)}v_1(z) \frac{w(s+x)}{w(x)}\theta(s)\overline{\bd T_z\theta}(s) e^{-2\pi i \eta \frac{F^{-1}(s+x)}{w(x)}}~ds\right|~d\eta ~dz}\\
       & \leq 2 C^{3/2} V_2(1) \int_\RR \int_\RR |w(s)^{3/2}v_1(s)\theta(s)||\sqrt{w(z-s)}v_1(z-s)\overline{\bd T_z\theta}(s)|~ds~dz\\
       & \leq 2 C^{3/2} V_2(1) \|\theta\|_{\bd L^2_{w_2}(\RR)}\|\theta\|_{\bd L^2_{w_1}(\RR)} \int_\RR (1+|z|)^{-1-\epsilon}~dz < \infty,
     \end{split}
   \end{equation}
   where the derivation follows the steps of the proof of Lemma \ref{lem:cznconvergence}.
   
   For the case $|\eta|\geq 1$, our estimate is based on Eq. \eqref{eq:pthorder}. 
   \begin{equation}\label{eq:largeeta}
     \begin{split}
       \lefteqn{\int_\RR\int_{\RR\setminus]-1,1[} V_2(\eta) \left|\int_\RR \sqrt{w(z)}v_1(z) \mathbf{D}^{p+2}_{w,x,\eta}\left(\frac{w(\cdot+x)}{w(x)}\theta\overline{\bd T_z\theta}\right)(s) e^{-2\pi i \eta \frac{F^{-1}(s+x)}{w(x)}}~ds\right|~d\eta ~dz}\\
       & \overset{\text{Cor.\ref{cor:Dnestimate}}}{\leq} C_{p+1}\int_\RR\int_{\RR} \frac{V_2(1+|\eta|)}{(2\pi(1+|\eta|)^{p+2}} \int_\RR \sqrt{w(z)}v_1(z)w(-s)^{p+1} \sum_{k=1}^{p+2}|\left(\theta \bd T\theta\right)^{(k)}(s)|~ds~d\eta ~dz\\
       & \overset{\text{Lem.\ref{lem:cznconvergence}}}{\leq} (p+2)\int_\RR \sqrt{w(z)}v_1(z) c_{p+1}(z) \int_{\RR} \frac{V_2(1+|\eta|)}{(2\pi(1+|\eta|)^{p+2}}~d\eta ~dz < \infty.       
     \end{split}
   \end{equation}
   To obtain finiteness, we used Lemma \ref{lem:cznconvergence} and $v_2\in\mathcal O\left((1+|\cdot|)^p \right) \Rightarrow V_2\in\mathcal O\left((1+|\cdot|)^p \right)$.
   Combine Eqs. \eqref{eq:smalleta} and \eqref{eq:largeeta} to prove the assertion. A more precise statement is obtained using the estimate in the proof of Lemma \ref{lem:cznconvergence}:
   \begin{equation}
     I_{\theta,F,m}(x,\xi) \leq \left(2V_2(1) \|\theta\|_{\bd L^2_{w_2}(\RR)}\|\theta\|_{\bd L^2_{w_1}(\RR)} + (p+2)E \max_{j=0}^{p+2} \left( \|\theta^{(j)}\|_{\bd L^2_{w_3}(\RR)}\|\theta^{(p-j+2)}\|_{\bd L^2_{w_1}(\RR)} \right)  \right) \tilde{C}Z_{\epsilon},
   \end{equation}
   with $E:= \int_{\RR} \frac{V_2(1+|\eta|)}{(2\pi(1+|\eta|)^{p+2}}~d\eta$ and $Z_{\epsilon} = \int_\RR (1+|z|)^{-1-\epsilon}~dz$. This completes the proof.
\end{proof}

We now have a set of conditions on $F$, $m$ and $\theta$ that guarantee $K_{\theta,F}\in\AAm$ and therefore allow the construction of a set of (generalized) coorbit spaces by applying Theorems \ref{thm:resanddual} and \ref{thm:coorbits}. It is easy to see that Theorem \ref{thm:mainres1} is just a special case of Theorem \ref{thm:thetainAm}.

\begin{proof}[Proof of Theorem \ref{thm:mainres1}]
  Set $m_2 = 1$ to see that the conditions on $F$ and $m$ imply the conditions of Theorem \ref{thm:thetainAm}. Furthermore, note that $\theta\in\mathcal C^\infty_c$ implies $\theta\in \bd L^2_{\sqrt{w}}$ and conditions (a-c). If furthermore $w=(F^{-1})'$ and $v_2$ are polynomial, then $\theta\in\mathcal S$ is sufficient for $\theta\in \bd L^2_{\sqrt{w}}$ and to imply conditions (a-c). Therefore, the result follows immediately from Theorem \ref{thm:thetainAm}.
\end{proof}

If $K_{\theta_1,F},K_{\theta_2,F}\in\AAm$ and additionally $\theta_1,\theta_2 \in \bd L^2_{\sqrt{w}}(\RR)\cap \LtR$, then we can also apply Proposition \ref{pro:mixedkern}, concluding that $\mathcal{G}(\theta_1,F)$ and $\mathcal{G}(\theta_2,F)$ give rise to the same coorbit spaces.

\begin{proposition}\label{pro:protoindep}
  If $\theta_1,\theta_2 \in \bd L^2_{\sqrt{w}}(\RR)\cap \LtR$ are such that $K_{\theta_1,F},K_{\theta_2,F}\in\AAm$, then 
  \begin{equation}
    K_{\theta_1,\theta_2,F} := K_{\mathcal{G}(\theta_1,F),\mathcal{G}(\theta_2,F)} \in\AAm.
  \end{equation}
  Hence, $\text{Co}(\mathcal{G}(\theta_1,F),Y) = \text{Co}(\mathcal{G}(\theta_2,F),Y)$ for all $Y$ that satisfy Eq. \eqref{eq:BScond}.
\end{proposition}
\begin{proof}
  Recall that $\AAm$ is an algebra, i.e. $K_1\circ K_2\in\AAm$, if $K_1,K_2\in\AAm$. By definition of the algebra multiplication,
  \begin{equation}
    \left( K_{\theta_1,F}\circ K_{\theta_2,F} \right)(x,y,\xi,\omega) = \int_D\int_\RR K_{\theta_1,F}(x,z,\xi,\eta)K_{\theta_2,F}(z,y,\eta,\omega)~d\eta~dz.
  \end{equation}  
  Insert the definition of $K_{\theta_1,F},K_{\theta_2,F}$ to obtain 
  \begin{equation}
  \begin{split}
    \left( K_{\theta_1,F}\circ K_{\theta_2,F} \right) (x,y,\xi,\omega) & = A_{\theta_1}^{-1}A_{\theta_2}^{-1}\int_D\int_\RR \overline{V_{\theta_1,F}g^0_{x,\xi}(z,\eta)}V_{\theta_2,F}g^1_{y,\omega}(z,\eta)~d\eta~dz\\
    & = A_{\theta_1}^{-1}A_{\theta_2}^{-1}\langle V_{\theta_2,F} g^2_{y,\omega},V_{\theta_1,F}g^1_{x,\xi} \rangle,
  \end{split}
  \end{equation}
  where $A_{\theta_1},A_{\theta_2}$ are the frame bounds of $\mathcal G(\theta_1,F)$, resp. $\mathcal G(\theta_2,F)$. Since $\theta_1,\theta_2 \in \bd L^2_{\sqrt{w}}(\RR)\cap \LtR$, we can apply the orthogonality relations Theorem \ref{thm:orthrel}, to obtain 
  \begin{equation}
    \langle V_{\theta_2,F}g^2_{y,\omega},V_{\theta_1,F}g^1_{x,\xi} \rangle = \langle \theta_1, \theta_2 \rangle \langle g^2_{y,\omega},g^1_{x,\xi} \rangle.
  \end{equation}
  In other words, 
  \begin{equation}
    K_{\theta_1,\theta_2,F} = \frac{A_{\theta_1}A_{\theta_2}}{\langle \theta_1,\theta_2 \rangle} K_{\theta_1,F}\circ K_{\theta_2,F} 
  \end{equation}
  and therefore $K_{\theta_1,\theta_2,F}\in\AAm$.
\end{proof}

\begin{corollary}
  In the setting of Theorem \ref{thm:mainres1}, the coorbit spaces generated by $\mathcal{G}(\theta,F)$, $\theta\in \mathcal C^\infty_c(\RR)$ (respectively $\theta\in\mathcal{S}(\RR)$) are independent of the particular choice of $\theta\in \mathcal C^\infty_c(\RR)$ ($\theta\in\mathcal{S}(\RR)$).
\end{corollary}
\begin{proof}
  Follows immediately from Proposition \ref{pro:protoindep}, since $C^\infty_c(\RR)\subset\mathcal{S}(\RR)\subset \bd L^2_{\sqrt{w}}(\RR)\cap \LtR$.
\end{proof}

We can now construct the coorbits of an abstract, solid Banach space $Y$ with respect to $\mathcal G(\theta,F)$, provided $Y$ satisfies Eq. \eqref{eq:BScond}. Let us provide some examples of such spaces and suitable combinations of weight functions $m$ and warping functions $F$.

Fix $1\leq p,q\leq \infty$ and choose a continuous weight function $v: D\times\RR \rightarrow \RR^+$. Then by Schur's test~\cite{samarah2005schur}, the 
 weighted mixed norm space $\bd L^{p,q}_{v}(D\times\RR)$ satisfies Eq. \eqref{eq:BScond} with 
    \begin{equation}
      m_v(x,y,\xi,\omega) := \max \left\{\frac{v(x,\xi)}{v(y,\omega)},\frac{v(y,\omega)}{v(x,\xi)}\right\}
    \end{equation}
If $v$ that is also bounded away from $0$ (resp. bounded above), then $v_{y,\omega}(x,\xi):= m(x,y,\xi,\omega)$ and $v$ (resp. $v_{y,\omega}^{-1}$ and $v$) are equivalent weights, for any fixed $(y,\omega)\in D\times \RR$. 

Let additionally $v$ be such that there is an equivalent tensor weight $\tilde{v}(x,\xi):=\tilde{v}_1(x)\tilde{v}_2(\xi)$, i.e. there are $C_1,C_2>0$ such that $C_1 \tilde{v} \leq v \leq C_2 \tilde{v}$. Then $m_v$ and $m_{\tilde{v}}$ are equivalent and we can try to apply Theorem \ref{thm:thetainAm} with regards to $\mathcal{A}_{m_v} = \mathcal{A}_{m_{\tilde{v}}}$. If $D=\RR$ and $\tilde{v}_2$ is a polynomial weight, then condition (ii) in Theorem \ref{thm:thetainAm} is satisfied for some $p\in\NN$. Otherwise only $\tilde{v}_2\equiv 1$ is valid. 

For $\tilde{v}_1$ however, we require $\tilde{v}_1\circ F^{-1}$ to be $v_1$-moderate for some symmetric, submultiplicative weight $v_1$. Without loss of generality, we can assume $v_1(x)=e^{a|x|}$ for some $a\geq 0$.

\begin{example}[Polynomial weights:]
Let $F_0(x) = \log(x)$, $F_1(x)=\sgn(x)\log(1+|x|)$ and $\tilde{v}_1(x) = (1+|x|)^p$, for some $p>0$. Then 
\begin{equation*}
 \tilde{v}_1(F_0^{-1}(x)) = (1+e^x)^p \leq 2e^{p|x|} \text{ and } \tilde{v}_1(F_1^{-1}(x)) = e^{p|x|}.
\end{equation*}
Similarly, if $F_{2,l}(x) = \sgn(x)\left((1+|x|)^l-1\right)$, $l\in [0,1[$, then 
\begin{equation*}
  \tilde{v}_1(F_{2,l}^{-1}(x))\in\mathcal O\left(1+|x|^{p/l}\right).                                                                                 
\end{equation*}
Consequently, $\tilde{v}_1\circ F^{-1}$ is $v_1$-moderate in all those cases. Hence, polynomial weights $\tilde{v}_1$ satisfy condition (i) in Theorem \ref{thm:thetainAm} for Examples \ref{ex:wavelet1}, \ref{ex:ERB1} and \ref{ex:alpha1}.

However, only in the case of $F_{2,l}$ is $\tilde{v}_1(F_{2,l}^{-1}(x))$ itself polynomial and $\theta\in\mathcal S$ a sufficient condition on $\mathcal{G}(\theta,F)$ to satisfy Theorem \ref{thm:thetainAm}. For $F_0,F_1$, $\theta\in \mathcal C^\infty_c$ is sufficient.
\end{example}

\begin{example}[Subexponential weights:]
Let $F_0,F_1$ and $F_{2,l}$ be as in the previous example, but $\tilde{v}_1(x) = e^{|x|^{\alpha}}$, for some $0 <\alpha\leq 1$. Then 
\begin{equation*}
 \tilde{v}_1(F_0^{-1}(x)) = e^{e^{\alpha x}} \text{ and } \tilde{v}_1(F_1^{-1}(x)) = e^{(e^{|x|}-1)^\alpha},
\end{equation*}
both of which are obviously not $v_1$-moderate. 

On the other hand, 
\begin{equation*}
  \tilde{v}_1(F_{2,l}^{-1}(x)) = e^{|F^{-1}_{2,l}(x)|^\alpha}, \text{ where } |F^{-1}_{2,l}(x)|^\alpha \in \Theta(1+|x|^{\alpha/l}). 
\end{equation*}
Hence, $\tilde{v}_1$ satisfies condition (i) if and only if $\alpha \leq l^{-1}$. Moreover, $\theta\in\mathcal S$ is never sufficient for $\mathcal{G}(\theta,F)$ to satisfy Theorem \ref{thm:thetainAm}, while $\theta\in \mathcal C^\infty_c$ always is.
\end{example}
    

\section{Discrete frames and atomic decompositions}\label{sec:decomps}

  We will now construct moderate, admissible coverings (see Definition \ref{def:modadmcover}) and show that families of covers and a canonical choice of $\Gamma$ exist, such that the associated $\Gamma$-oscillation converges to $0$ in $\AAm$, i.e. 
  \begin{equation*}
      \|\text{osc}_{\mathcal U^\delta,\Gamma}\|_{\AAm} \overset{\delta \rightarrow 0}{\rightarrow} 0\text{ and } C_{m,\mathcal U^\delta} \overset{\delta \rightarrow 0}{\rightarrow} C < \infty,
  \end{equation*}
  for any admissible warping function $F$ and sufficiently smooth, quickly decaying prototype $\theta$.

  Consequently, the discretization machinery provided by Sections \ref{ssec:discret} and \ref{sec:genosc} can be put to work, providing atomic decompositions and Banach frames with respect to $\mathcal G(\theta,F)$ and the family of coverings $\mathcal U^\delta$, $\delta > 0$. Let us first define a prototypical family of coverings induced by the warping function.
  
  \begin{definition}\label{def:inducedcover}
    Let $F$ be a warping function that satisfies Eq. \eqref{eq:quasisubmult}. Define $\mathcal U^\delta_F = \{U^\delta_{l,k}\}_{l,k\in\ZZ}$, $\delta>0$ by
    \begin{equation}\label{eq:deltacover}
      U_{l,k}^{\delta}:= I_{F,l}^{\delta} \times \left[\frac{\delta^2 k}{|I_{F,l}^{\delta}|},\frac{\delta^2 (k+1)}{|I_{F,l}^{\delta}|}\right], \text{ where } I_{F,l}^{\delta}:= \left[F^{-1}(\delta l),F^{-1}(\delta (l+1))\right].
    \end{equation}
    We call $\mathcal U^\delta_F$ the \emph{$F$-induced $\delta$-cover}. For all $\delta > 0$, $\mathcal U^\delta_F$ is a moderate, admissible covering with $\mu(U_{l,k}^{\delta}) = \delta^2$, where $\mu$ is the standard Lebesgue measure.
  \end{definition}

  Let us state our second main result.
  
  \begin{theorem}\label{thm:mainres2}
    Let $F~:~D\mapsto \RR$ be a warping function with $w = (F^{-1})'\in\mathcal{C}^1(\RR)$, such that for all $x,y\in\RR$:
   \begin{equation}
      \frac{w(x+y)}{w(x)w(y)} \leq C<\infty \text{ and } \left|\frac{w'}{w}\right|(x)\leq D_1<\infty.
   \end{equation}
   Furthermore, let $\mathcal U^\delta_F$ the induced $\delta$-cover and $m_1:D\rightarrow \RR$ such that $m_1\circ F^{-1}$ is $v_1$-moderate, for a symmetric, submultiplicative weight function $v_1$ and define $m(x,y,\xi,\omega) = \max\left\{\frac{m_1(x)}{m_1(y)},\frac{m_1(y)}{m_1(x)}\right\}$.
   Then $\sup_{l,k\in \ZZ}\sup_{(x,\xi),(y,\omega)\in U_{l,k}^\delta}  m(x,y,\xi,\omega) < \infty$ and 
   \begin{equation}
     \text{osc}_{\mathcal U^\delta_F,\Gamma}\in \AAm,\quad \text{ for all } \theta\in\mathcal C^\infty_c,\ \delta>0,
   \end{equation}
   where $\Gamma(x,y,\xi,\omega) = e^{2\pi i (\omega-\xi)x}$.
   If furthermore $w,v_1\in \mathcal O\left((1+|\cdot|)^p\right)$ for some $p\in\RR^+$, then
   \begin{equation}
     \text{osc}_{\mathcal U^\delta_F,\Gamma}\in \AAm,\quad \text{ for all } \theta\in \mathcal S,\ \delta>0.
   \end{equation} 
   For sufficiently small $\delta$, there are constants $C>0$ and $C_{m,\mathcal U^\delta_{F}} \geq \sup_{k,l\in \ZZ}\sup_{(x,\xi),(y,\omega)\in U_{l,k}^\delta}  m(x,y,\xi,\omega)$ such that $C_{m,\mathcal U^\delta_{F}} < C$. Furthermore,
   \begin{equation}
     \forall\ \theta\in\mathcal C^\infty_c (\theta\in \mathcal S),\ \epsilon>0\ \exists\ \delta>0 \text{ such that } \|\text{osc}_{\mathcal U^\delta_F,\Gamma}\| < \epsilon.
   \end{equation}
  \end{theorem}
  
  Similar to the previous section, Theorem \ref{thm:mainres2} is a special case of a more general result with weaker conditions on $F$, $m$ and $\theta$. And once more, the proof of that result requires some amount of preparation. First, we take a closer look at the sets $Q_{y,\omega}$ from the definition of the $\Gamma$-oscillation kernel.
  
  \begin{lemma}\label{lem:coverings}
       Let $F$ be a warping function that satisfies Eq. \eqref{eq:quasisubmult} and $\mathcal U^\delta_F$ the induced $\delta$-cover. 
       For all $(y,\omega)\in D\times \RR$ and all $\delta > 0$,
       \begin{equation}
         Q_{y,\omega} = \bigcup_{\substack{(l,k), \text{ s.t.}\\ (y,\omega)\in U_{l,k}^{\delta}}} U_{l,k}^{\delta} \subseteq I_y \times (\omega+J_{y}),
       \end{equation}
       where
       \begin{equation}
           I_y = \left[F^{-1}(F(y)-\delta),F^{-1}(F(y)+\delta)\right] \text{ and }
       \end{equation}
       \begin{equation}
           J_y = \left[-\frac{C\delta w(\delta)}{w(F(y))},\frac{C\delta w(\delta)}{w(F(y))} \right].         
       \end{equation}
  \end{lemma}
  \begin{proof}
       Assume that $(y,\omega)\in U_{l,k}^\delta$, then in turn 
       \begin{equation}
         I_{F,l}^{\delta} = \left[F^{-1}(\delta l),F^{-1}(\delta (l+1))\right] 
			  \subseteq \left[F^{-1}(F(y)-\delta),F^{-1}(F(y)+\delta)\right]. 
       \end{equation}
       Furthermore,
       \begin{equation}
         \left[\frac{\delta^2 k}{|I_{F,l}^{\delta}|},\frac{\delta^2 (k+1)}{|I_{F,l}^{\delta}|}\right] \subseteq \left[\omega-\frac{\delta^2}{|I_{F,l}^{\delta}|},\omega+\frac{\delta^2}{|I_{F,l}^{\delta}|}\right].
       \end{equation}
       Assume $D=\RR^+$. Since $w$ is nondecreasing and satisfies Eq. \eqref{eq:quasisubmult}, 
       \begin{equation}
         |I^\delta_{F,l}| \geq F^{-1}(F(y))-F^{-1}(F(y)-\delta) \geq \delta w(F(y)-\delta) \geq \delta C^{-1}w(F(y))/w(\delta),
       \end{equation}
       where we applied the FTC. Therefore 
       \begin{equation}\label{eq:omegaint_estimate}
         \left[\frac{\delta^2 k}{|I_{F,l}^{\delta}|},\frac{\delta^2 (k+1)}{|I_{F,l}^{\delta}|}\right] \subseteq \left[\omega-\frac{C\delta w(\delta)}{w(F(y))},\omega+\frac{C\delta w(\delta)}{w(F(y))}\right].
       \end{equation}
       This completes the proof for $D=\RR^+$. For $D=\RR$ and $|y|\geq \delta$, Eq. \eqref{eq:omegaint_estimate} holds by the same argument. For $|y|<\delta$, the FTC yields
       \begin{equation}
         |I^\delta_{F,l}| \geq \delta w(0).
       \end{equation}
       On the other hand $w(0) \geq C^{-1}w(F(y))/w(F(y))\geq C^{-1}w(F(y))/w(\delta)$, showing that Eq. \eqref{eq:omegaint_estimate} holds for all $y\in\RR$. 
     \end{proof}
     
     The next two results are concerned with a certain family of operators. At this point, their definition might seem arbitrary, but their purpose will become clear once we investigate $\text{osc}_{\mathcal U^\delta_F,\Gamma}$ more closely. In particular, we investigate whether those operators approximate the identity in a suitable way.  
     
     \begin{lemma}\label{lem:continuouseye}
       Let $X = \bd L^p_w(\RR)$, $1\leq p \leq \infty$, where $w=(F^{-1})'$ for some warping function $F$ satisfying Eq. \eqref{eq:quasisubmult}. For all $y\in\RR, \epsilon \geq 0$, let $\Eye ~:~ X\rightarrow X$ be the operator defined by 
       \begin{equation}
         \Eye f(t) = f(t)e^{2\pi i \epsilon \frac{F^{-1}(t+y)-F^{-1}(y)}{w(y)}}\ \text{a.e. , for all } f\in X.
       \end{equation}
       The following hold:
       \begin{itemize}
        \item[(i)] If $\supp(f)\subseteq [-\delta,\delta]$ and $0 \leq \epsilon \leq \frac{1}{2C\delta w(\delta)}$, then 
          \begin{equation}
            \|f-\Eye f\|_X \leq 2\delta \sqrt{2-2\cos\left(2\pi \epsilon C\delta w(\delta)\right)}\|f\|_X.
          \end{equation}
        \item[(ii)] The map $\epsilon \mapsto \sup\limits_{y\in\RR} \|f-\Eye f\|_X$ is continuous at $\epsilon = 0$ for any fixed $f\in X$.
       \end{itemize}
     \end{lemma}
     \begin{proof}
       We only provide the proof for $p<\infty$, the proof for $p=\infty$ is analogous. In order to prove (i), note that 
       \begin{equation}\label{eq:Lpnormequality}
         \|f-\Eye f\|_{\bd L^p_w(\RR)}^p = \int_{-\delta}^\delta |1-e^{2\pi i\epsilon\frac{F^{-1}(t+y)-F^{-1}(y)}{w(y)}}|^p|f(t)|^p w(t)^p\; dt.
       \end{equation}
       By Eq. \eqref{eq:quasisubmult} and the FTC, 
       \begin{equation}
         \left|\frac{F^{-1}(t+y)-F^{-1}(y)}{w(y)}\right| \leq C\delta w(\delta),
       \end{equation}
       where we used $w$ nondecreasing ($D=\RR^+$), respectively or nondecreasing on $\RR^+$ and odd ($D=\RR$). Therefore, 
       \begin{equation}
         \left|1-e^{2\pi i\epsilon\frac{F^{-1}(t+y)-F^{-1}(y)}{w(y)}}\right| = \sqrt{2-2\cos\left(2\pi \epsilon\frac{F^{-1}(t+y)-F^{-1}(y)}{w(y)}\right)} \leq \sqrt{2-2\cos\left(2\pi \epsilon C\delta w(\delta)\right)}, 
       \end{equation}       
       for all $0 \leq \epsilon \leq \frac{1}{2C\delta v(\delta)}$. Inserting into Eq. \eqref{eq:Lpnormequality} proves (i). For proving (ii), 
       \begin{equation}
         \sup\limits_{y\in\RR} \|f-\Eye f\|_X \overset{\epsilon\rightarrow 0}{\longrightarrow} 0,
       \end{equation}
       is sufficient as the result then follows by the triangle inequality, since $E_{y,\epsilon_0+\epsilon_1} = E_{y,\epsilon_1}E_{y,\epsilon_0}$, for all $\epsilon_0,\epsilon_1 \geq 0$. For any $f\in X$, we can construct a sequence $(f_n)_{n\in\NN}\subset X$ of compactly supported functions, i.e. $\supp(f_n)\subseteq [-\delta_n,\delta_n]$ with $\delta_n \overset{n\rightarrow \infty}{\longrightarrow} \infty$, converging in norm to $f$. For every $n\in\NN$, 
       \begin{equation}
        \begin{split}
         \sup\limits_{y\in\RR}  \|f-\Eye f\|_X & \leq \|f-f_n\|_X + \sup\limits_{y\in\RR}  (\|f_n-\Eye f_n\|_X + \|\Eye f_n-\Eye f\|_X)\\
         & = 2\|f-f_n\|_X + \sup\limits_{y\in\RR} \|f_n-\Eye f_n\|_X. 
        \end{split}
       \end{equation}
       By (i) however, $\|f_n-\Eye f_n\|_X$ is bounded uniformly independent of $y$, provided $\epsilon$ is small enough. Consequently,
       \begin{equation}
         \forall\ \epsilon_0 > 0\ \exists\ (n,\epsilon)\in \NN\times \RR^+ \text{ such that } \|f-f_n\|_X < \epsilon_0/3 \text{ and }\|f_n-\Eye f_n\|_X < \epsilon_0/3,
       \end{equation}
       completing the proof.
     \end{proof}
     
     The next result clarifies the stability of $\Eye$ when combined with differentiation.
 
     \begin{lemma}\label{lem:continuousdiffeye}
     Let $X = \bd L^p_{\tilde{w}}(\RR)$, $1\leq p \leq \infty$, with some weight function $\tilde{w}$, and $w=(F^{-1})'\in\mathcal C^{n-1}(\RR)$ for some warping function $F$ satisfying Eq. \eqref{eq:quasisubmult} and $\Eye$ as defined in Lemma \ref{lem:continuouseye}. If there are $D_k>0$ such that $|w^{(k)}/w|(s)\leq D_k<\infty, \text{ for all } 0\leq k \leq n-1$ and $\theta\in\mathcal C^{n}(\RR)\cap \bd L^2_{\sqrt{w}}(\RR)$ satisfies 
     \begin{equation}
         \theta^{(n)}\in X \text{ and }
     \end{equation}
     \begin{equation}
         \theta^{(k)}w^l \in X \text{ for all } 0\leq k \leq n-1 \text{ and } 1\leq l\leq n-k,          
     \end{equation}
     then $(\Eye \theta)^{(n)} \in X$ and the map $\epsilon \mapsto \sup\limits_{y\in\RR} \|(\theta-\Eye \theta)^{(n)}\|_{X}$ is continuous at $\epsilon = 0$.
     \end{lemma}
     \begin{proof}
       Assume 
       \begin{equation}\label{eq:diffeyen}
         (\Eye \theta)^{(n)} = \Eye \theta^{(n)} + \Eye \left(\sum_{k=0}^{n-1}\theta^{(k)}\sum_{l=1}^{n-k} \left(2\pi i\epsilon\right)^{l} \bd T_{-y} P_{n,k,l}\right),
       \end{equation}
       where 
       \begin{equation}
         P_{n,k,l}(s) := w(y)^{-l} \sum_{\sigma\in\Sigma_{n,k,l}} C_{\sigma} \prod_{m=1}^l w^{(\sigma_m)}(s),
       \end{equation}
       with $\Sigma_{n,k,l}\subseteq \{\sigma = (\sigma_1,\ldots,\sigma_l)~:~ \sigma_m \in (0,\ldots,n-k-1)\}$ and some $C_\sigma \in\RR$.
       By the conditions on $w$, 
       \begin{equation}
         \begin{split}
         |P_{n,k,l}(s+y)| & \leq \sum_{\sigma\in\Sigma_{n,k,l}} |C_{\sigma}| \prod_{m=1}^l \frac{|w^{(\sigma_m)}(s+y)|}{w(y)} \leq \sum_{\sigma\in\Sigma_{n,k,l}} |C_{\sigma}| \prod_{m=1}^l D_{\sigma_m}\frac{w(s+y)}{w(y)}\\
         & \leq \sum_{\sigma\in\Sigma_{n,k,l}} C^l |C_{\sigma}| \prod_{m=1}^l D_{\sigma_m}w(s) \leq C^l 
         \left(\max_{\sigma\in\Sigma_{n,k,l}} |C_\sigma|\right)\left(\max_{j=0}^{n-1}D_j^l\right) w(s)^l \sum_{\sigma\in \Sigma_{n,k,l}} 1\\
         & = C_{n,k,l}w(s)^l.
         \end{split}
       \end{equation}
       Since all the sums in Eq. \eqref{eq:diffeyen} are finite, there is some $\tilde{C}>0$ such that
       \begin{equation}
         \begin{split}
         \|(\theta-\Eye\theta)^{(n)}\|_X & \leq \|\theta^{(n)}-\Eye\theta^{(n)}\|_X + \|\Eye \left(\sum_{k=0}^{n-1}\theta^{(k)}\sum_{l=1}^{n-k} \left(2\pi i\epsilon\right)^l \bd T_{-y} P_{n,k,l}\right)\|_X\\
         & \leq \|\theta^{(n)}-\Eye\theta^{(n)}\|_X + \tilde{C}\max_{k=0}^{n-1} \max_{l=1}^{n-k} \|\theta^{(k)}w_0^l\|_X |\epsilon|,
         \end{split}
       \end{equation}
       for all $0\leq \epsilon \leq (2\pi)^{-1}$. For $\epsilon\rightarrow 0$, the first term converges to $0$ by Lemma \ref{lem:continuouseye}. To complete the proof, we need to show that Eq. \eqref{eq:diffeyen} holds. Clearly,  
       \begin{equation}\label{eq:diffeye}
         (\Eye \theta)' = \Eye\theta' + \Eye\left( 2\pi i\epsilon \frac{\bd T_{-y} w}{w(y)} \theta \right),
       \end{equation}
       proving Eq. \eqref{eq:diffeyen} for $n=1$. Assume it holds for $n-1$, then 
       \begin{equation}
         (\Eye \theta)^{(n)} = \left(\Eye \theta^{(n-1)} + \Eye \left(\sum_{k=0}^{n-2}\theta^{(k)}\sum_{l=1}^{n-k-1} \left(2\pi i\epsilon\right)^l \bd T_{-y} P_{n-1,k,l}\right) \right)'.
       \end{equation}
       We now consider the derivative of each term separately. For the first term, invoke Eq. \eqref{eq:diffeye} for $\theta^{(n-1)}$. All the other terms are of the form
       \begin{equation}
         C_{\sigma}\left(2\pi i\epsilon\right)^l w(y)^{-l}\left(\Eye \left( \theta^{(k)} \prod_{m=1}^l \bd T_{-y} w^{(\sigma_m)} \right)\right)', 
       \end{equation}
       for some $\sigma\in\Sigma_{n-1,k,l}$, $0\leq k\leq n-2$ and $1\leq l\leq n-k-1$. Apply Eq. \eqref{eq:diffeye} to $\theta^{(k)} \prod_{m=1}^l \bd T_{-y} w^{(\sigma_m)}$ to obtain
       \begin{equation}
         \begin{split}
	    C_{\sigma}\left(2\pi i\epsilon\right)^l w(y)^{-l} & \Bigg( \Eye\left( \theta^{(k+1)} \prod_{m=1}^l \bd T_{-y}  w^{(\sigma_m)} \right) \\
	    & + 2\pi i\epsilon  \Eye\left(\theta^{(k)} \frac{\bd T_{-y} w \prod_{m=1}^l \bd T_{-y}  w^{(\sigma_m)}}{w(y)} \right) + R_{n-1,k,l}\Bigg), 
	 \end{split}
       \end{equation}
       with 
       \begin{equation}
          R_{n-1,k,l} = \sum_{m=1}^l \Eye\left(\theta^{(k)} \bd T_{-y}w^{(\sigma_m+1)}\prod_{j\in\{1,\ldots,l\}\setminus\{m\}} \bd T_{-y} w^{(\sigma_j)}  \right).
       \end{equation}
       Reorder everything by the appearing derivative of $\theta$ to complete the proof.
     \end{proof}
     
     We are now ready to prove the central statements, which we will split into two more compact results.
     
     \begin{proposition}\label{pro:boundedCU}
       Let $F:D\rightarrow \RR$ be a warping function satisfying Eq. \eqref{eq:quasisubmult}. Let $p\in\NN$ if $D=\RR$ and $p=0$ if $D=\RR^+$. Let $\mathcal U^\delta_F$ the \emph{induced $\delta$-cover} and
       \begin{equation} 
         m(x,y,\xi,\omega) := \max \left\{\frac{m_1(x)m_2(\xi)}{m_1(y)m_2(\omega)},\frac{m_1(y)m_2(\omega)}{m_1(x)m_2(\xi)}\right\},
       \end{equation}
       with weight functions $m_1,m_2$ that satisfy
       \begin{itemize}
         \item[(i)] $m_1\circ F^{-1}$ is $v_1$-moderate, for a symmetric, submultiplicative weight function $v_1$ and
         \item[(ii)] $m_2$ is $v_2$-moderate, for a symmetric, submultiplicative weight function $v_2
         \in \mathcal O\left((1+|\cdot|)^p \right)$.
       \end{itemize}
       Then, for a suitable constant $\tilde{C}>0$
       \begin{equation}
        \sup_{l,k\in\ZZ}\sup_{(x,\xi),(y,\omega)\in U_{l,k}^{\delta}} m(x,y,\xi,\omega) \leq C_{m,\mathcal U^{\delta}_F}:= \tilde{C}v_1(\delta)V_2, \text{ where } V_2:= \begin{cases} 
		\max\limits_{|u|\leq 1} v_2(u/w(0)) &\text{ if } D=\RR, \\
		1 &\text{ if } D=\RR^+.                                                                                                        
	\end{cases}
     \end{equation}
     In particular, $C_{m,\mathcal U^{\delta}_F} < C$, for sufficiently small $\delta$ and some constant $C$.
     \end{proposition}
     \begin{proof}
       Clearly,
       \begin{equation}
         m(x,y,\xi,\omega) \leq \max\left\{\frac{m_1(x)}{m_1(y)},\frac{m_1(y)}{m_1(x)}\right\}\max\left\{\frac{m_2(\xi)}{m_2(\omega)},\frac{m_2(\omega)}{m_2(\xi)}\right\}.
       \end{equation}
       Since $x,y\in [F^{-1}(\delta l),F^{-1}(\delta (l+1))]$ and $m_1\circ F^{-1}$ is $v_1$-moderate, 
       while $\xi,\omega\in \left[\frac{\delta^2 k}{|I_{F,l}^{\delta}|},\frac{\delta^2 (k+1)}{|I_{F,l}^{\delta}|}\right]$ and $m_2$ is $v_2$-moderate,
       \begin{equation}\label{eq:newstuff}
	  m(x,y,\xi,\omega) \leq \tilde{C}_1\tilde{C}_2v_1(\delta)v_2(\delta/|I_{F,l}^{\delta}|).
       \end{equation}
       Note that we can, without loss of generality, assume $v_1,v_2$ to be nondecreasing away from zero. Otherwise, obtain \eqref{eq:newstuff} by sufficiently increasing $\tilde{C}_1,\tilde{C}_2$.
       
       If $D=\RR^+$, set $v_2 = 1$ to complete the proof. If $D=\RR$, note that the FTC yields
       \begin{equation}
         \delta\min_{t\in [\delta l,\delta (l+1)]} w(t) \leq |I^\delta_{F,l}| \leq \delta\max_{t\in [\delta l,\delta (l+1)]} w(t),
       \end{equation}
       where $w = (F^{-1})'$. Therefore $v_2(\delta/|I_{F,l}^{\delta}|) \leq \max_{|u|\leq 1} v_2(u/w(0)) = V_2$.
     \end{proof}
     
     \begin{theorem}\label{thm:osckernel}
     Let $F:D\rightarrow \RR$ be a warping function such that $w=(F^{-1})'\in \mathcal{C}^{p+1}(\RR)$ satisfies Eq. \eqref{eq:quasisubmult} and $|w^{(k)}/w|\leq D_k$ for some constants $D_k > 0$, $k=0,\ldots,p+1$, where $p\in\NN$ if $D=\RR$ and $p=0$ if $D=\RR^+$.
     Furthermore, let 
     \begin{equation} 
	m(x,y,\xi,\omega) := \max \left\{\frac{m_1(x)m_2(\xi)}{m_1(y)m_2(\omega)},\frac{m_1(y)m_2(\omega)}{m_1(x)m_2(\xi)}\right\},
     \end{equation}
     with weight functions $m_1,m_2$ that satisfy
     \begin{itemize}
	\item[(i)] $m_1\circ F^{-1}$ is $v_1$-moderate, for a symmetric, submultiplicative weight function $v_1$ and
	\item[(ii)] $m_2$ is $v_2$-moderate, for a symmetric, submultiplicative weight function $v_2
	\in \mathcal O\left((1+|\cdot|)^p \right)$.
     \end{itemize}
       
       If $\theta\in\mathcal C^{p+2}(\RR)\cap \bd L^2_{\sqrt{w}}(\RR)$ satisfies
       \begin{itemize}
	  \item[($\tilde{a}$)] $\theta,\ w(-s)^j w(s)^l \theta^{(k+1)}\in\mathcal C_0$, for all $0\leq l\leq p-k+2$, $0\leq k\leq j$ and $0\leq j\leq p+1$,
	  \item[(b)] $\theta\in \bd L^2_{w_1}(\RR)\cap \bd L^2_{w_2}(\RR)$, with $w_1(s) = v_1(s)(1+|s|)^{1+\epsilon}w(-s)^{1/2}$, $w_2(s) = w_1(-s)w(s)$ and
	  \item[($\tilde{c}$)] $\theta^{(p+2)},\ \theta^{(k)}w^l \in \bd L^2_{w_1}(\RR)\cap \bd L^2_{w_3}(\RR)$ for $1\leq l\leq p-k+2$ and $0\leq k \leq p+1$, with $w_3(s) = w_1(-s)w(-s)^{p+1}$,
       \end{itemize}
       then the following hold: For $\Gamma(x,y,\xi,\omega) = e^{-2\pi i(\xi-\omega)x}$ and all $\delta > 0$
       \begin{equation}
         \|\text{osc}_{\mathcal U^{\delta}_F,\Gamma} \|_{\AAm} < \infty.
       \end{equation}
       For any $\epsilon > 0$, there is a $\delta_0 > 0$ such that for all $0<\delta\leq \delta_0$
       \begin{equation}
         \|\text{osc}_{\mathcal U^{\delta}_F,\Gamma} \|_{\AAm} < \epsilon.
       \end{equation}
     \end{theorem}
     \begin{proof}
       We begin by showing that $\widehat{g_{x,\xi}}-e^{-2\pi i(\xi-\omega)x}\widehat{g_{y,\omega}}$ can be rewritten as  
       \begin{equation}
         \sqrt{F'(x)}(\bd T_x \widetilde{\theta})\circ F(t)e^{2\pi i\xi t},
       \end{equation}
       the warping of a function $\widetilde{\theta}$, depending on $(x,y,\xi,\omega)$.
       We have
       \begin{equation}
         \begin{split}
           \lefteqn{\widehat{g_{x,\xi}}-e^{-2\pi i(\xi-\omega)x}\widehat{g_{y,\omega}}(t)}\\
           & = \sqrt{F'(x)}\theta(F(t)-F(x))e^{-2\pi i\xi t} - \sqrt{1/w(F(y))}\theta(F(t)-F(y))e^{-2\pi i\omega t}e^{-2\pi i(\xi-\omega)x}\\
           & = \sqrt{F'(x)}e^{-2\pi i\xi t} \left((\bd T_{F(x)}\theta)- \sqrt{\frac{w(F(x))}{w(F(y))}}(\bd T_{F(y)}\theta)e^{2\pi i(\xi-\omega)(F^{-1}(\cdot)-x)}\right)(F(t))\\
           & = \sqrt{F'(x)}e^{-2\pi i\xi t} \left(\bd T_{F(x)}\left(\theta- \sqrt{\frac{w(F(x))}{w(F(y))}}(\bd T_{F(y)-F(x)}\theta)e^{2\pi i(\xi-\omega)(F^{-1}(\cdot+F(x))-x)}\right)\right)(F(t))\\
           & = \sqrt{F'(x)}e^{-2\pi i\xi t} \left(\bd T_{F(x)}\theta_{x,y,\epsilon_{1}}\right)(F(t)).
         \end{split}
       \end{equation}
       Here
       \begin{equation}
         \theta_{x,y,\epsilon_{1}} := \theta-\sqrt{\frac{w(F(x))}{w(F(y))}}\operatorname{E}_{F(x),\epsilon_{1}}(\bd T_{\epsilon_0}\theta),
       \end{equation}
       with $\epsilon_0 = F(y)-F(x)$ and $\epsilon_{1} = w(F(x))(\xi-\omega)$. Now assume $(y,\omega)\in Q_{x,\xi}$, where $Q_{x,\xi}$ is as given in Definition \ref{def:genosckern}. By Lemma \ref{lem:coverings}, 
       \begin{equation}
         \epsilon_0 \leq \delta \text{ and } \epsilon_{1} \leq C w(\delta)\delta.
       \end{equation}
       
       Define 
       \begin{equation}
         O_{\theta,F,\delta,1}(x,\xi) := \int_D\int_\RR m(F^{-1}(x),y,\xi,\omega) \text{osc}_{\mathcal U^{\delta}_F,\Gamma}(F^{-1}(x),y,\xi,\omega)~d\omega~dy, \text{ for all } x,\xi\in\RR,
       \end{equation}
       and 
       \begin{equation}
         O_{\theta,F,\delta,2}(y,\omega) := \int_D\int_\RR m(x,y,\xi,\omega) \text{osc}_{\mathcal U^{\delta}_F,\Gamma}(x,y,\xi,\omega)~d\xi~dx, \text{ for all } y,\omega\in\RR.
       \end{equation}
       Clearly, $Q_{y,\omega}$ is compact and $\langle g_{x,\xi},g_{y,\omega}-\Gamma(y,z,\omega,\eta)g_{z,\eta} \rangle$ is continuous as a function in $(z,\eta)\in D\times\RR$.
       Therefore, denoting $\sigma = (x,y,\xi,\omega)$, there are $(z_{\sigma},\eta_{\sigma}), (\tilde{z}_{\sigma},\tilde{\eta}_{\sigma}) \in Q_{y,\omega}$
       , such that 
       \begin{equation}\label{eq:O1}
         O_{\theta,F,\delta,1}(x,\xi) = \int_D\int_\RR m(F^{-1}(x),y,\xi,\omega) |\langle g_{F^{-1}(x),\xi},g_{y,\omega}-\Gamma(y,z_{\sigma},\omega,\eta_{\sigma})g_{z_{\sigma},\eta_{\sigma}} \rangle|~d\omega~dy,
       \end{equation}
       \begin{equation}\label{eq:O2}
         O_{\theta,F,\delta,2}(y,\omega) = \int_D\int_\RR m(x,y,\xi,\omega) |\langle g_{x,\xi},g_{y,\omega}-\Gamma(y,\tilde{z}_{\sigma},\omega,\tilde{\eta}_{\sigma})g_{\tilde{z}_{\sigma},\tilde{\eta}_{\sigma}} \rangle|~d\xi~dx.
       \end{equation}
       Let us investigate Eq. \eqref{eq:O1} first. Performing the same substitution steps as for obtaining Eq. \eqref{eq:amclass}, we get
       \begin{equation}
         O_{\theta,F,\delta,1}(x,\xi) = \int_\RR \int_\RR \left| \int_\RR C_x(z) \tilde{m}(x,z,\xi,\eta)\frac{w(s+x)}{w(x)} \theta(s)\overline{\bd T_{z}\theta_{y,z_{\sigma},\epsilon_{1,\sigma}}(s)}e^{-2\pi i \eta\frac{F^{-1}(s+x)}{w(x)}} ~ds\right|~d\eta~dz,
       \end{equation}
       where $\epsilon_{1,\sigma} = w(F(y))(\omega-\eta_{\sigma})$. If $\theta, \theta_{y,z_{\sigma},\epsilon_{1,\sigma}}$ satisfy conditions (a-c) from Theorem \ref{thm:thetainAm}, then
       \begin{equation}\label{eq:O1max}
       \begin{split}
         \lefteqn{O_{\theta,F,\delta,1}(x,\xi)}\\
         & \leq \esssup_{(y,\omega)\in D\times\RR} \left(2V_2(1) \|\theta\|_{\bd L^2_{w_2}(\RR)}\|\theta_{y,z_{\sigma},\epsilon_{1,\sigma}}\|_{\bd L^2_{w_1}(\RR)} + (p+2)E \max_{j=0}^{p+2} \left( \|\theta^{(j)}\|_{\bd L^2_{w_3}(\RR)}\|\theta_{y,z_{\sigma},\epsilon_{1,\sigma}}^{(p-j+2)}\|_{\bd L^2_{w_1}(\RR)} \right)  \right) \tilde{C}Z_{\epsilon},
         \end{split}
   \end{equation}
   with $E:= \int_{\RR} \frac{V_2(1+|\eta|)}{(2\pi(1+|\eta|)^{p+2}}~d\eta$ and $Z_{\epsilon} = \int_\RR (1+|z|)^{-1-\epsilon}~dz$. Similarly, 
   \begin{equation}\label{eq:O2max}
   \begin{split}
         \lefteqn{O_{\theta,F,\delta,2}(y,\omega)}\\
         & \leq \sup_{(z_0,\eta_0)\in Q_{y,\omega}} \sup_{|\epsilon_{1,\sigma}|\leq Cw(\delta)\delta} \left(2V_2(1) \|\theta_{y,z_0,\epsilon_{1,\sigma}}\|_{\bd L^2_{w_2}(\RR)}\|\theta\|_{\bd L^2_{w_1}(\RR)} + (p+2)E \max_{j=0}^{p+2} \left( \|\theta_{y,z_0,\epsilon_{1,\sigma}}^{(j)}\|_{\bd L^2_{w_3}(\RR)}\|\theta^{(p-j+2)}\|_{\bd L^2_{w_1}(\RR)} \right)  \right) \tilde{C}Z_{\epsilon},
         \end{split}
   \end{equation}
   where $\epsilon_{1,\sigma} = w(F(y))(\omega-\eta_{\sigma})$. Clearly, the conditions on $\theta$ imply conditions (a-c) for $\theta$ itself (simply set $l=0$ in ($\tilde{\text{a}}$) and ($\tilde{\text{c}}$)), but $\theta_{y,z_0,\epsilon_{1,\sigma}}$ requires a closer investigation.
   
   Recall 
   \begin{equation}
     \theta_{y,z_0,w(F(y))(\omega-\eta_{\sigma})} =  \theta-\sqrt{\frac{w(F(y))}{w(F(z_0))}}\operatorname{E}_{F(y),w(F(y))(\omega-\eta_{\sigma})}(\bd T_{F(z_0)-F(y)}\theta).
   \end{equation}
   The estimates in the proof of Lemma \ref{lem:continuousdiffeye} hold pointwise, therefore
   \begin{equation}
     |(\Eye (\bd T_{x}\theta))^{(n)}|(s)\leq |(\bd T_{x}\theta)^{(n)}|(s) + \tilde{C} \max_{k=0}^{n-1} \left( |\bd T_{x}\theta^{(k)}|(s) \max_{l=1}^{n-k}(2\pi|\epsilon|)^l w(s)^l \right),
   \end{equation}
   for all $0\leq n\leq p+2$. Consequently, condition ($\tilde{\text{a}}$) on $\theta$ implies condition (a) for $\theta-\tilde{D} \Eye (\bd T_x\theta)$ for all $\tilde{D}> 0$, $x,\epsilon\in\RR$. In particular, $\theta_{y,z_0,w(F(y))(\omega-\eta_{\sigma})}$ satisfies condition (a) for all $(y,\omega),(z_0,\eta_0)\in D\times\RR$. For any $v$-moderate weight function $w_0$ and $X=\bd L^p_{w_0}$ for some $1\leq p\leq \infty$, 
   \begin{equation}
     \|\Eye (\bd T_{x}\theta)\|_X = \|\bd T_{x}\theta\|_X \leq C_v v(x) \|\theta\|_X.
   \end{equation}
   Hence condition (b) on $\theta$ implies condition (b) for $\theta-\tilde{D} \Eye (\bd T_x\theta)$ for all $\tilde{D}> 0$, $x,\epsilon\in\RR$. Finally, by Lemma \ref{lem:continuousdiffeye}, condition ($\tilde{\text{c}}$) on  implies condition (c) for $\theta-\tilde{D} \Eye (\bd T_x\theta)$ and we have shown that the RHS of Eqs. \eqref{eq:O1max} and \eqref{eq:O2max} are valid expressions. 
   
   Let $k\leq p+2$, $\delta>0$ and $(z_0,\eta_0)\in Q_{y,\omega}$. Using Lemma \ref{lem:coverings}, we obtain $|F(z_0)-F(y)| < \delta$ and by the FTC and $w'/w \leq D_1$
   \begin{equation}
     \frac{w(F(z_0))}{w(F(y))} \leq \frac{w(F(y))+\delta D_1w(\delta)w(F(y))}{w(F(y))} = 1+\delta D_1w(\delta),
   \end{equation}
   and similarly $(1+\delta D_1w(\delta))^{-1} \leq \frac{w(F(z_0))}{w(F(y))}$. Recall $|\epsilon_{1,\sigma}|\leq Cw(\delta)\delta$ and use Lemma \ref{lem:continuousdiffeye} to see that for $X= \bd L^2_{w_i}$, $i=1,2,3$,
   \begin{equation}
     w(F(z_0))/w(F(y)) \rightarrow 1,~ (\operatorname{E}_{F(y),\epsilon_{1,\sigma}}\theta)^{(k)} \rightarrow \theta^{(k)} \text{ and } \bd (T_{F(z_0)-F(y)}\theta)^{(k)} \rightarrow \theta^{(k)},
   \end{equation}
   for $\delta \rightarrow 0$, uniformly over $(y,\omega)\in D\times\RR$, $(z_0,\eta_0)\in Q_{y,\omega}$.
   Therefore, their composition is also uniformly continuous and, together with \eqref{eq:diffeyen},
   \begin{equation}
     \max_{k=0}^{p+2} \esssup_{(y,\omega)\in D\times\RR} \sup_{(z_0,\eta_0)\in Q_{y,\omega}}  \sup_{|w(F(y))(\eta_0-\omega)|\leq Cw(\delta)\delta} \|\theta_{y,z_0,w(F(y))(\eta_0-\omega)}^{(k)} \|_X \overset{\delta \rightarrow 0}{\rightarrow} 0.
   \end{equation}
   This shows that Eqs. \eqref{eq:O1max} and \eqref{eq:O2max} tend to $0$ for $\delta \rightarrow 0$, proving the second assertion.
   In particular, we can apply the estimates in the proofs of Lemmata \ref{lem:continuouseye} and \ref{lem:continuousdiffeye} to obtain 
   \begin{equation}
     \max_{k=0}^{p+2} \esssup_{(y,\omega)\in D\times\RR} \sup_{(z_0,\eta_0)\in Q_{y,\omega}} \sup_{|w(F(y))(\eta_0-\omega)|\leq Cw(\delta)\delta}\left\| \theta_{y,z_0,w(F(y))(\eta_0-\omega)}^{(k)}\right\|_X < \infty, \text{ for all } \delta > 0,     
   \end{equation}
   showing that Eqs. \eqref{eq:O1max} and \eqref{eq:O2max} are finite, for all $\delta > 0$, proving the first assertion.   
   \end{proof}
   
   The statements we have just proven specify a set of conditions on $F$, $m$ and $\theta$ such that we can construct atomic decompositions and Banach frames by invoking Theorem \ref{thm:discret2}. That the conditions of Theorem \ref{thm:mainres2} imply the conditions in Theorem \ref{thm:osckernel} and Proposition \ref{pro:boundedCU} is easily seen.

   \begin{proof}[Proof of Theorem \ref{thm:mainres2}]
     Analogous to the proof of Theorem \ref{thm:mainres1}, but use Theorem \ref{thm:osckernel} and Proposition \ref{pro:boundedCU} instead of Theorem \ref{thm:thetainAm}. 
   \end{proof}
   
   \begin{remark}
     Although we only state Theorems \ref{thm:mainres2} and \ref{thm:osckernel}, as well as Proposition \ref{pro:boundedCU}, for the induced $\delta$-cover, it is easily seen that any covering $\widetilde{\mathcal U}$ that satisfies Lemma \ref{lem:coverings}, for $\delta > 0$ small enough, guarantees $\|\text{osc}_{\widetilde{\mathcal U},\Gamma}\|_{\AAm} < \epsilon$ and $\sup_{l,k\in\ZZ}\sup_{(x,\xi),(y,\omega)\in \widetilde{U_{l,k}}} m(x,y,\xi,\omega) \leq C_{m,\widetilde{\mathcal U}}<\infty$. If $\epsilon>0$ is in turn small enough, then Theorem \ref{thm:discret2} can be applied, providing atomic decompositions and Banach frames with respect to $\widetilde{\mathcal U}$.
   \end{remark}
  
   \section{Conclusion and Outlook}
   
   In this contribution, we introduced a novel family of time-frequency representations tailored to a wide range of nonlinear frequency scales. We have shown that the resulting integral transforms are invertible and produce continuous functions on phase space. Under mild restrictions on the chosen  frequency scale, every such representation gives rise to a full family of (generalized) coorbit spaces. Furthermore, through a minor, but important generalization to existing discretization results in generalized coorbit theory, we are able to prove that atomic decompositions and Banach frames can be constructed in a natural way, provided that the system is discretized respecting suitable density conditions.
   
   There still are many open questions regarding the finer structure of coorbit space theory for warped time-frequency representations, e.g. whether the generated coorbit spaces coincide with some known localization spaces. Since the warping functions $F(x) = x$ and $F(x) = \log(x)$ yield short-time Fourier and wavelet transforms, the associated coorbit spaces can be expected to coincide with their classical counterpart. Furthermore, the close relationship between the $\alpha$-transform and the warping functions discussed in Examples \ref{ex:ERB1} and \ref{ex:alpha1} suggests a connection to $\alpha$-modulation spaces that requires closer study.
   
   Another interesting question is the relation between the spaces $\{\theta\in L^2_{\sqrt{w}} ~:~ K_{\theta,F}\in\AAm\}$ and $\{g\circ F^{-1}~:~ g\in\mathcal H^1_v\}$. Clearly, the first space is contained in the second, since $\mathcal G(\theta,F)\subset \mathcal H^1_v$, but at this point it is unclear whether the inclusion is strict.
   
   The construction of (Hilbert space) frames by means of discrete warped time-frequency systems is covered in the first part of this series~\cite{howi14}. Therein generalizations of classical necessary and sufficient frame conditions, previously known to hold for Gabor and wavelet systems,  are recovered. A special focus in~\cite{howi14} is the construction of tight frames with bandlimited elements, also illustrated through a series of examples.        
      
  \section*{Acknowledgment}
  This work was supported by the Austrian Science Fund (FWF)
  START-project FLAME (``Frames and Linear Operators for Acoustical
  Modeling and Parameter Estimation''; Y 551-N13) and the Vienna Science and Technology Fund (WWTF) 
  Young Investigators project CHARMED (``Computational harmonic analysis of high-dimensional biomedical data''; VRG12-009).
  

 \bibliographystyle{amsplain}
 \providecommand{\bysame}{\leavevmode\hbox to3em{\hrulefill}\thinspace}
\providecommand{\MR}{\relax\ifhmode\unskip\space\fi MR }
\providecommand{\MRhref}[2]{%
  \href{http://www.ams.org/mathscinet-getitem?mr=#1}{#2}
}
\providecommand{\href}[2]{#2}


 \end{document}